\documentclass{amsart}
\usepackage{amsmath}
\usepackage{amssymb}
\usepackage[all]{xy}
\usepackage{comment}
\usepackage{color}

\theoremstyle{plain}
\newtheorem{thm}{Theorem}[section]
\newtheorem{thm*}{Theorem}[section]
\newtheorem{cor}[thm]{Corollary}
\newtheorem{prop}[thm]{Proposition}

\newtheorem{lemma*}{Lemma}

\newtheorem{question}[thm]{Question}
\newtheorem{hypothesis}[thm]{Hypothesis}

\theoremstyle{definition}
\newtheorem{defn}[thm]{Definition}
\newtheorem{remark}[thm]{Remark}

\newtheorem{note}[thm]{Notation}

\newtheorem*{remark*}{Remark}

\newtheorem{ex}[thm]{Example}
\newtheorem{notation}[thm]{Notation}

\newcommand{\bbH}{\mathbb H}

\newcommand{\cM}{\mathcal M}
\newcommand{\cN}{\mathcal N}
\newcommand{\cU}{\mathcal U}

\def\Spec{\operatorname{Spec}\nolimits}

\def\Proj{\operatorname{Proj}\nolimits}

\newcommand{\bG}{\mathbb G}

\newcommand{\cO}{\mathcal O}
\newcommand{\bA}{\mathbb A}

\newcommand{\cF}{\mathcal F}
\newcommand{\bP}{\mathbb P}
\newcommand{\bU}{\mathbb U}
\newcommand{\bZ}{\mathbb Z}
\newcommand{\bF}{\mathbb F}

\newcommand{\cC}{\mathcal C}

\newcommand{\cL}{\mathcal L}

\newcommand{\cE}{\mathcal E}

\newcommand{\fg}{\mathfrak g}
\newcommand{\fgp}{\mathfrak g_{\mathbb F_p}}

\newcommand{\fu}{\mathfrak u}

\newcommand{\gl} {\mathfrak {gl}}

\newcommand{\ol}{\overline}
\newcommand{\ul}{\underline}

\def\Spec{\operatorname{Spec}\nolimits}
\def\sl2{\operatorname{SL_{2(2)}}\nolimits}
\def\Ga2{\operatorname{\mathbb G_{\rm a(2)}}\nolimits}

\newcommand{\bE}{\mathbb E}

\newcommand{\bu}{\bullet}

\date\today

\begin{document}

 \title[Geometric invariants for representations of finite groups]{Geometric invariants of representations of finite groups}
 
 \author[ Eric M. Friedlander]
{Eric M. Friedlander$^{*}$} 

\address {Department of Mathematics, University of Southern California,
Los Angeles, CA}
\email{ericmf@usc.edu}

\thanks{$^{*}$ partially supported by the Simons Foundation}



\begin{abstract} 
J. Pevtsova and  the author constructed a ``universal $p$-nilpotent operator" for an infinitesimal group
scheme $G$ over a field $k$ of characteristic $p > 0$ which led to coherent sheaves on the scheme of 
1-parameter subgroups of $G$ associated to a $G$-module $M$.  Of special interest is the fact that these 
coherent sheaves are vector bundles if $M$ is of constant Jordan type.  In this paper, we provide similar invariants for 
a finite group $\tau$ which recover the invariants earlier obtained for elementary abelian $p$-groups.  To do this, 
we replace the analogue of 1-parameter subgroups by a refined version of equivalence
classes of $\pi$-points for $k\tau$.   More generally, we provide a construction of vector bundles for the semi-direct
product $G\rtimes \tau$ of an infinitesimal group scheme $G$ and a finite grooup $\tau$.

A major motivation for this study is to further our understanding of the relationship
between representations of $\bG(\bF_p)$ and $\bG_{(r)}$
associated to a finite dimensional rational $\bG$-module $M$, where $\bG$ is a reductive group with 
$r$-th Fobenius kernel $\bG_{(r)}$.
Using vector bundles, we extend and sharpen earlier results comparing support varieties.  
\end{abstract}

\maketitle


\section{Introduction}

As shown to us by Quillen \cite{Q1}, \cite{Q2}, the 
spectrum $\Spec H^\bu(\tau,k)$ of the even dimensional cohomology of a finite group $\tau$ is an important invariant,
where $k$ is a field of characteristic $p$ for some prime $p$ dividing the order of $\tau$.
(If $p=2$, we set $H^\bu(\tau,k) = H^*(\tau,k)$, the full cohomology algebra.)  For $\tau \ = \ E \ \simeq \ \bZ/pZ^{\times r}$,
an elementary abelian $p$-group of rank $r$, there is a natural map $S^\bu((J_E/J_E)^*) \ \to \ H^\bu(E,k)$ which
induces a homeomorphism of prime ideal spectra, $\Spec H^\bu(E,k) \ \to \ \bA_{J_E/J_E^2}$.  Here, $S^\bu(V)$ denotes 
the commutative, graded $k$ algebra generated by $V$ and $\bA_{J_E/J_E^2}
 \equiv \Spec S^\bu((J_E/J_E)^*)$ is the affine space associated to the vector space $J_E/J_E^2$.  Much of this paper
confronts the challenge posed by the observation that there does not appear to be a ``canonical" way to associate
 to a point in $\bA_{J_E/J_E^2}$ a $p$-nilpotent operator on $kE$-modules.

In contrast, if $\fg$ is a $p$-restricted Lie algebra with restricted enveloping  algebra $\fu(\fg)$, $\Spec H^\bu(\fu(\fg),k)$ maps 
homeomorphically (i.e., is a map of schemes which is a ``$p$-isogoeny" in the sense of \cite{Q1})
 to the variety $\cN_p(\fg)$ of $p$-nilpotent elements of $\fg$ (see \cite{FPar}, \cite{SFB2}).  For any $\fu(\fg)$-module
$M$, the action of $\fg$ on $M$ affords the action of $X$ on $M$ for any $X \in \cN_p(\fg)$.   In particular, if 
$\fg = \fg_a^{\oplus r}$, the commutative Lie algebra of dimension $r$ over $k$ with trivial $p$-restriction, then
$\fu( \fg_a^{\oplus r}) \ = \ S^\bu(\epsilon)/(X^p, X \in \epsilon)$, where $\epsilon$ is the underlying vector space of 
$g_a^{\oplus r}$.  In this case, $S^\bu(\epsilon^*) \to H^\bu(\fu( \fg_a^{\oplus r}),k)$ induces an isomorphism of reduced algebras
yielding the $p$-isogeny $\Spec H^\bu(\fu( \fg_a^{\oplus r}),k) \ \stackrel{\simeq}{\to} \ \bA_{\epsilon}$.  
Thus, if we choose an isomorphism $kE \ \simeq \fu(\fg_a^{\oplus r})$ (thereby identifying $J_E/J_E^2$ with $\epsilon$),
then we can associate $p$-nilpotent operators on $kE$-modules to points of $\bA_{\epsilon}$.  This is
the approach taken by various authors to associate vector bundles on projective spaces (namely, $\Proj S^\bu(\epsilon^*)$)
to $kE$-modules of constant Jordan type (see \cite{FP4}, \cite{BP}, \cite{Ben}).

For a given elementary abelian $p$-group $E$, a choice of $k$-vector space splitting $J_E/J_E^2 \to J_E$ of the quotient map 
$J_E \to J_E^2$ determines an isomorphism $S^\bu(J_E/J_E^2) \to kE$ and thus an isomorphism $\fu(\fg_a^{\oplus r})
 \ \simeq \ kE$ (where we have identified the underlying $k$-vector space of $\fg_a^{\oplus r}$ with $J_E/J_E^2$).
A reasonable splitting is given once a choice of generating set $\{ g_1,\ldots,g_r \}$ of $E$ is made.  Such a 
choice determines the generating set $\{ x_1 \ = \ g_1 -1, \ldots,  x_r \ = \ g_r -1 \}$ of the ideal $J_E$ whose
image $\{ \ol x_1,\ldots,\ol x_r \}$ is a basis for $J_E/J_E^2$.   
This is  the approach 
used in the past when investigating vector bundles associated to certain $kE$-modules as in \cite{Ben}, \cite{BP}.

Our ``solution" to this lack of naturality of associating a $p$-nilpotent operator to an element of $J_E/J_E^2$ is to replace 
$S^\bu((J_E/J_E^2)^*)$ by $S^\bu(J_E^*)$, thereby avoiding the necessity of a choice of splitting.    To the reader familiar with support varieties, our construction is somewhat analogous to replacing 
J. Carlson's cyclic shifted subgroups \cite{C} in the definition of support varieties for $kE$-modules by 
equivalence classes of $\pi$-points considered by Pevtsova and the author for any finite group $\tau$ (see \cite{FP2}).    
We introduce a universal $p$-nilpotent operator which is {\it natural} with respect to $E$, $\Theta_E \in S^\bu(J_E^*) \otimes kE$.
The operator  $\Theta_E $determines a $p$-nilpotent operator 
$\Theta_{E,M}: S^\bu(J_E^*) \otimes M \to  S^\bu(J_E^*) \otimes M$ for each $kE$-module $M$.
Taking kernels, images, and cokernels of powers of the projective variant $\bP \Theta_{E,M}: \cO_{\bP_{J_E}}\otimes M 
\to  \cO_{\bP_{J_E}}(1) \otimes M$ of $\Theta_{E,M}$ provides coherent sheaves 
on $\bP_{J_E} \equiv \Proj S^\bu(J_E^*)$.  If $M$ is a 
module of constant Jordan type for $E$, then these coherent sheaves are vector bundles.
Proposition \ref{prop:elem} shows how to recover from this construction the earlier construction of vector bundles 
introduced in \cite{FP4} using a choice of splitting of $J_E \to J_E/J_E^2$. 

In order to consider an arbitrary finite group $\tau$, we introduce in Section \ref{sec:tau} the commutative $k$-algebra 
$A_\tau$ defined as the inverse limit of polynomial algebras $S^\bu(J_E^*)$ indexed by elementary abelian 
$p$-subgroups $E \subset \tau$ and consider the affine $k$-scheme 
$X_\tau \ =  \ \Spec A_\tau \ \simeq \  \varinjlim_E \bA_{J_E}$ equipped with the action of $\tau$ given by 
conjugation.   We also consider
$X_\tau^{(2)} \simeq \  \varinjlim_E \bA_{J_E^2}$ and $Y_\tau \simeq \  \varinjlim_E \bA_{J_E/J_E^2}$.  The naturality
of $\Theta_E$ enables the definition of the 
$p$-nilpotent operator $\Theta_\tau \ \in \ A_\tau \otimes k\tau$.  This determines the $p$-nilpotent, $\tau$-invariant $A_\tau$-endomorphism $\Theta_{\tau,M}$ of $M \otimes A_\tau$ and its projective variant $\bP \Theta_{\tau,M}$
for any $k\tau$-module $M$.

Associated to a finite dimensional $k\tau$-module $M$, we construct $\tau$-equivariant
coherent sheaves on $\bP X_\tau$ in Theorem \ref{thm:projbundle} by taking kernels, images, and cokernels of 
iterates of $\bP \Theta_{\tau,M}$.  For $k\tau$-modules of constant
$j$-rank, these coherent sheaves restrict to $\tau$-equivariant vector bundles on $\bU_\tau \ \equiv \ 
\bP X_\tau \backslash \bP X_\tau^{(2)}$.   Such $\tau$-equivariant vector bundles can be identified with vector bundles on the quotient stack $[\bU_\tau/\tau]$.  The quotient scheme $(\bP Y_\tau)/\tau$ is 
$p$-isogenous (and thus homeomorphic) to $\Proj H^\bu(\tau,k)$.  Thus, our construction of vector bundles on $[\bU_\tau/\tau]$
is analogous to the construction in \cite{FP1} and
\cite{FP2} of vector bundles on a scheme $p$-isogenous to $\Proj H^\bu(G,k)$ for an infinitesimal group scheme. 

In Section \ref{sec:Ktheory}, we observe that the class in $K_0(\bU_\tau)$  of these bundles
(forgetting the action of $\tau$) is somewhat
accessible since the natural projection $p_\tau: \bU_\tau \to \bP Y_\tau$ induces an isomorphism on 
Grothendieck groups of vector bundles.

For a reductive group $\bG$ defined over $\bF_p$ and a finite dimensional rational $\bG$-module $M$, we revisit 
in Section \ref{sec:rational} the challenge of comparing the restrictions of $M$ to $\bG(\bF_p)$ and $\bG_{(r)}$ (the
$r$-th Frobenius kernel of $\bG$).  We relate our support variety for the finite group $\bG(\bF_p)$ to the support
variety for $\bG_{(1)}$ (the $p$-nilpotent cone $\cN_p(\fg)$ of $\fg = Lie(\bG)$) and then ``lift" this relation
to the support variety $V_r(\bG)$ of $\bG_{(r)}$.   Our investigation refines and extends the original work of 
Z. Lin and D. Nakano \cite{LN} proving ``Parshall's Conjecture" comparing the support varieties for 
$\bG(F_p)$ and $\bG_{(1)}$ and further work by various authors (e.g., \cite{CLN} and \cite{F2}).

In Section \ref{sec:Gtau}, we combine the construction given in \cite{FP4}
of coherent sheaves associated to finite dimensional $G$-modules
for $G$ infinitesimal  with our construction for finite groups. The result is a construction which yields 
vector bundles associated to modules
of constant Jordan type (or, more generally, of constant $j$-rank) for any finite group scheme
of the form $G \rtimes \tau$, where
$G$ is an infinitesimal group scheme and $\tau$ is a finite group.  For $k$ algebraically closed, every finite group
scheme is of this form.

Throughout this paper, $p$ is a fixed prime, $k$ a field of characteristic $p$, and $\tau$ is a finite group of order
divisible by $p$.

We thank Jesse Burke, Marc Hoyois, Julia Pevtsova, and Paul Sobaje for various conversations which have 
helped shape this paper.
\vskip .2in


\section{$X_{\tau}$  for a finite group $\tau$}
\label{sec:tau}

We begin this section with the following useful theorem of D. Ferrand which we will apply to obtain
various schemes associated to $\tau$.  The reader familiar with cohomological support varieties can interpret these
schemes as geometric variations of $\Spec H^\bu(\tau,k)$.

\begin{thm} \cite[5.1]{Fer}
\label{thm:ferrand}
Consider a Cartesian square of commutative $k$-algebras
\begin{equation}
\label{diag:cart}
\begin{xy}*!C\xybox{%
\xymatrix{
A_{1,2} & A_1  \ar[l] \\
A_2 \ar[u]^q & A \ar[l]  \ar[u]
 }
}\end{xy}
\end{equation}
(i.e.,  $A \simeq  A_1 \times _{A_{1,2}} A_2$)
and assume that $q$ is surjective.  Then 
\begin{equation}
\label{diag:cocart}
\begin{xy}*!C\xybox{%
\xymatrix{
\Spec A_{1,2} \ar[d]_q \ar[r] & \Spec A_1 \ar[d] \\
\Spec A_2 \ar[r] & \Spec A
 }
}\end{xy}
\end{equation}
is cocartesian (i.e., a push-out square) in the category of $k$-schemes
(and $q$ is a closed immersion).
\end{thm}

The following proposition extends Theorem \ref{thm:ferrand} to certain colimits realized as iterated push-out squares.

\vskip .1in
\begin{prop}
\label{prop:limits-colimits}
Let $\cC$ be a category whose set of objects is a finite collection $\{ V_i,  \ i \in I\}$ of subspaces of a fixed 
finite dimensional $k$-vector space $V$ and whose maps are inclusions (commuting with the fixed $k$-linear
embeddings $V_i \subset V$).   Let $\bA_{V _i }\  \equiv \ \Spec S^\bu(V_i^*)$ denote the affine $k$-space 
naturally associated to $V_i$.
\begin{enumerate}
\item
The finitely generated, commutative $k$-algebra $\varprojlim_\cC S^\bu(V_i^*)$ satisfies the property that 
$$\Spec \varprojlim_\cC (S^\bu(V_i^*)) \ \simeq \ \varinjlim_\cC \bA_{V _i},$$
where the limit is taken in the catgory of $k$-algebras and the colimit is taken in the category of $k$-schemes.
\item
The nilradical of $\varprojlim_\cC S^\bu(V_i^*)$ is trivial, so that $\varinjlim_\cC \bA_{V _i}$ is reduced.
\item
For each $j \in J$, the  canonical maps $\bA_{V_j} \to \varinjlim_\cC \bA_{V _i}$ and 
$\varinjlim_\cC \bA_{V _j} \ \to \  \Spec S^\bu(V^*)$ are closed immersions.
\end{enumerate}
\end{prop}

\begin{proof}
We proceed by induction on the dimension of $V$, which we may assume to be spanned by the $V_i$; in other words, we may
assume that $V \simeq \varinjlim_\cC V_i$, where the limit is taken in the category of $k$-vector spaces.  
Let $V_0 \in \cC$ be chosen such that $V_0$ contains some element not in
the span $V^\prime$, the span of $\{ V_i, \ 0 \not= i \}$.   
If $V_0 = V$, then the assertions of the proposition are trivial, so we may assume that 
$V_0$ is a proper subspace of $V$.   Denote by $\cC^\prime$ the full subcategory of $\cC$ whose objects are those
of $\cC$ except $V_0$, and set $V^\prime \simeq \varinjlim_{\cC^\prime} V_i$.   

Observe that
\begin{equation}
\label{diag:cart1}
\begin{xy}*!C\xybox{%
\xymatrix{
\varprojlim_{\cC^\prime} S^\bu((V_i \cap V_0)^*)  & S^\bu(V_0^*)  \ar[l] \\
\varprojlim_{\cC^\prime} S^\bu(V_i^*)  \ar[u] & \varprojlim_{\cC} S^\bu(V_i^*)  \ar[l]  \ar[u]
 }
}\end{xy}
\end{equation}
is Cartesian.  By induction, the left vertical map of (\ref{diag:cart1}) is surjective which implies that the right vertical map
is surjective since (\ref{diag:cart1}) is Cartesian.  Consequently, Theorem \ref{thm:ferrand} plus induction 
implies assertion (1).  Equivalently, the following is a push-out square of schemes
\begin{equation}
\label{diag:cocart1}
\begin{xy}*!C\xybox{%
\xymatrix{
\varinjlim_{\cC^\prime} \bA_{V_i \cap V_0}  \ar[d] \ar[r]& \bA_{V_0}  \ar[d] \\
\varinjlim_{\cC^\prime} \bA_{V_i } \ar[r] & \varinjlim_{\cC} \bA_{V_i } 
 }
}\end{xy}
\end{equation}

Assertion (2) follows from the observation that the natural map $\varprojlim_\cC S^*(V_i) \ \to  \prod_{V_i \in \cC} S^*(V_i)$
is injective.

To prove the surjectivity of $S^\bu(V^*) \to \varprojlim_\cC S^\bu(V _i)$ as required
in assertion (3), we proceed once again by induction on the number of objects of $\cC$, so that the assertion is assumed 
valid for the proper subspace $V^\prime \subset V$.  Choose a basis of $V$ consisting of elements 
$x_1,\ldots,x_m \in V^\prime \backslash V_0$, elements $y_1,\ldots, y_n \in V^\prime \cap V_0$, and elements 
$z_1\ldots,z_\ell \in V_0 \backslash V^\prime$.  
Let the pair $(f(\ul x, \ul y), g(\ul y, \ul z)) \in \varprojlim_{\cC^\prime} S^\bu(V_i^*) \times S^\bu(V_0^*)$ 
restrict to  an element in $\varprojlim_{\cC} S^\bu((V_i\cap V_0)^*)$; thus,
$f(0,\ul y) = g(\ul y,0) \in \varprojlim_{\cC^\prime} S^\bu((V_i \cap V_0)^*)$.  Then $f(\ul x,\ul y) + g(\ul y,\ul z) - h(\ul y) \in
S^\bu(V^*)$ maps to $f(\ul x, \ul y), g(\ul y, \ul z) $, where $h(\ul y)$ is the polynomial $f(0,\ul y) = g(\ul y,0)$.
\end{proof}

Let $E \simeq \bZ/p^{\times r}$ be an elementary abelian $p$-group with identity element $e$, group algebra $kE$,
and augmentation ideal $J_E \ \subset \ kE$ with natural basis $\{g-e, e \not= g \in E\}$ as a $k$-vector space.  Our
strategy is to replace the $r$-dimensional $k$-vector space $J_E/J_E^2$ by the $p^r-1$-dimensional space
$J_E$ and/or the pair $(J_E, J_E^2)$.  

\vskip .1in
\begin{defn}
\label{defn:tau-category}
For any finite group $\tau$, we denote by $\cE(\tau)$ the category whose objects are elementary abelian $p$-subgroups
of $\tau$ and whose maps are inclusions.    If $E \in \cE(\tau)$ and $x\in \tau$, we denote by $E^x \subset \tau$
the subgroup consisting of elements $ xgx^{-1}, g \in E$ and let $c_x: \tau \to \tau$ denote conjugation 
by $x$ sending $g \in \tau$ to  $g^x \equiv xgx^{-1}$.   We also denote by $c_x: kE  \to kE^x$ the induced map
on group algebras. 

Conjugation determines the action
\begin{equation}
\label{eqn:tau-on-E}
\tau \times \cE(\tau) \ \to \ \cE(\tau), \quad (x,E) \mapsto E^x.
\end{equation}
\end{defn}

The following proposition introduces the affine scheme $X_\tau$ for a finite group $\tau$, 
specializing to
the affine space of dimension $p^r-1$, 
$\Spec S^\bu(J_E^*) \ \equiv \bA_{J_E}$, for $\tau$ equal to an elementary abelian $p$-group $E$ of rank $r$.

\begin{prop}
\label{prop:limits}
The algebra $A_\tau \ \equiv \ \varprojlim_{E \in \cE(\tau)} S^\bu(J_E^*)$ is a finitely generated commutative
$k$-algebra with a natural grading induced by the grading on each symmetric algebra $S^\bu(J_E^*)$.  
Moreover, the affine $k$-scheme
$$X_\tau\ \equiv  \ \Spec A_\tau \ = \ \Spec  (\varprojlim_{E \in \cE(\tau)} S^\bu(J_E^*))$$
equals the colimit in the category of schemes given by
$$X_\tau \simeq \ \varinjlim_{E \in \cE(\tau)} \bA_{J_E}, \quad \bA_{J_E} \equiv \Spec S^\bu(J_E^*).$$
Furthermore, $X_\tau$ is reduced (i.e., $A_\tau$ has no non-zero nilpotent elements).

A point of $X_\tau$ is a point of some affine space $ \bA_{J_E},$ and thus a $k$-rational point
can be viewed as an element of $J_E$ for some $E \in \cE(\tau)$.
\end{prop}

\begin{proof}
We apply Proposition \ref{prop:limits-colimits} with $\cC$ equal to the partially order set of vector spaces
$\{ J_E, E \in \cE_\tau \}$, where each $J_E$ is a subspace of the vector space colimit  $\varinjlim_{E \in \cE(\tau)} J_E$.  
This establishes most statements of the proposition.  The action of $\tau$ on $A_\tau$ arises from the action of $\tau$ on $\cE(\tau)$.
The natural grading on each $S^\bu(J_E^*)$
and the fact that the restriction map $S^\bu(J_E^*) \to S^\bu(J_{E^\prime}^*)$ is a graded map (i.e., preserves gradings)
for any $E^\prime \subset E$ provide a grading on $A_\tau$.

If $E_1,\ldots,E_s$ is
a list of the maximal elementary abelian $p$-subgroups of the finite group $\tau$, then $X_\tau$
is obtained from the disjoint union $\coprod_{j=1}^s \bA_{J_{E_j}}$ by identifying the images 
of $\bA_{J_{E_i \cap E_j}}$ in $\bA_{J_{E_i}}$ and in $\bA_{J_{E_j}}$.   In particular, $A_\tau$ is a 
subalgebra of the product algebra $\prod_{j=1}^s S^\bu(J_{E_j}^*)$ and thus has no non-zero nilpotent elements.

The last statement concerning the representation of a $k$-rational point of $X_\tau$ as an element of some $J_E$
is immediate from the description of $X_\tau$ as a colimit.
\end{proof}

The conjugation action of $\tau$  on $\cE(\tau)$ given by (\ref{eqn:tau-on-E}) determines
 an action of $\tau$ on $X_\tau$ as a $k$-scheme (or, equivalently, on $A_\tau$ as a $k$-algebra).
We proceed to make this explicit.

\begin{defn}
\label{defn:tau-action}
Let $E_1,\ldots,E_s$ be a list of the maximal elementary abelian $p$-subgroups of the finite group $\tau$
and represent an element  $f \in A_\tau \subset \prod_{j=1}^s S^\bu(J_{E_j}^*)$ by an $s$-tuple 
$\{ h_{E_j} \in S^\bu(J_{E_j}^*) \}$.  For $g \in \tau$,  $g \circ f$ is defined on $y \in J_{E_j} \subset X_\tau$ by
$$(g \circ f)(y) \ \equiv \ (g \circ \{ h_{E_j} \} )(y) \ = \ h_{E_{j^\prime}}(c_{g^{-1}}(y)), \quad E_{j^\prime} = (E_j)^{g^{-1}}.$$
\end{defn}

\vskip .1in

By replacing $E \mapsto J_E$ by $E \mapsto J_E/J_E^2$, we verify with only notational changes the following
analogue of Proposition \ref{prop:limits}.

\begin{prop}
\label{prop:EmodE2}
The algebra $B_\tau \ \equiv \ \varprojlim_{E \in \cE(\tau)} S^\bu((J_E/J_E^2)^*)$ is a finitely generated commutative
$k$-algebra with a natural grading induced by the grading on each symmetric algebra $S^\bu((J_E/J_E^2)^*)$.  
Moreover, the affine $k$-scheme
$$Y_\tau\ \equiv  \ \Spec B_\tau \ = \ \Spec  \varprojlim_{E \in \cE(\tau)} S^\bu((J_E/J_E^2)^*)$$
equals the colimit in the category of schemes given by
$$Y_\tau \simeq \ \varinjlim_{E \in \cE(\tau)} \bA_{J_E/J_E^2}, \quad \bA_{J_E/J_E^2} \equiv \Spec S^\bu((J_E/J_E^2)^*).$$

Furthermore, $Y_\tau$ is reduced.  The projection $J_E \twoheadrightarrow J_E/J_E^2$ determines the
$\tau$-equivariant map 
$$p_\tau: X_\tau \ \to \ Y_\tau.$$
\end{prop}

In considering vector bundles, it is more informative to consider the projectivizations $\bP X_\tau, \ \bP Y_\tau$
of $X_\tau, \ Y_\tau$ as 
constructed in the following proposition.   These are related by maps 
$$\bP X_\tau \ \hookleftarrow \bP X_\tau \backslash \bP X_\tau^{(2)}\ \to \ \bP Y_\tau$$
where 
\begin{equation}
\label{eqn:A2tau} 
A_\tau^{(2)} \ \equiv \ \varprojlim_{E\in \cE(\tau)} S^\bu((J_E^2)^*), \quad 
X_\tau^{(2)} \ \equiv \ \Spec A_\tau^{(2)}.
\end{equation}

\begin{prop}
\label{prop:projlimits}
Set \ $\bP X_\tau \ \equiv \ \Proj(A_\tau)$ \ and set \ $\bP Y_\tau \ \equiv \ \Proj(B_\tau)$.
These are colimits in the category of schemes
$$\bP X_\tau \ \simeq \ \varinjlim_{E \in \cE(\tau)} \bP X_E, 
\quad \bP Y_\tau \ \simeq \ \varinjlim_{E \in \cE(\tau)} \bP Y_E$$
equipped with an action of $\tau$ as in Defintion \ref{defn:tau-action}.
\end{prop}

\begin{proof}
We consider $\bP X_\tau$; the proof of the assertion for $\bP Y_\tau$ requires only notational changes.    
Let $0 \not= z \in \varinjlim_{E \in \cE(\tau)} (\bP X_E)(k)$ and let $E_0$ be the smallest element in $\cE_\tau$
such that $z \in (\bP X_{E_0})(k)$.   Let $F_{E_0} \in J_{E_0}^* \subset S^\bu(J_{E_0}^*)$ be some
homogeneous polynomial of degree 1 satisfying  the
conditions that $F_{E_0}(z) \not= 0$ whereas $(F_{E_0}^*)_{|J_{E^\prime}} = 0$ for all proper subgroups $E^\prime
\subset E_0$.    For example, we could take $F_z$ equal to the product indexed by $E^\prime \subset E_0$ each of whose
factors is a choice of linear form with zero locus $\bA_{E^\prime}$ in $\bA_{E_0}$.
We define $F_z \in A_\tau$ to be the
element (homogeneous, of degree 1) whose image in $S^\bu(J_E^*)$ is given by the image of 
$F_{E_0} $ under the composition 
$J_{E_0}^*\ \to \ (J_{E \cap E_0})^* \ \to \ J_E^*$, where the first map is induced by the homomorphism
$E\cap E_0 \to E_0$ and the second map is the dual of the map which sends the basis element $(g-e_0)$ 
of $J_E$ to $0$ if $g\notin E_0 \cap E$ and to the same named element of $J_{E_0} \cap J_E$ if 
$g \in E_0 \cap E$.  Thus, $F_z \in A_\tau \equiv \varprojlim_{E \in \cE(\tau)} S^\bu(J_E^*)$ maps to 
0 in $S^\bu(J_E^*)$ if and only if $z \notin E$.

Observe that $A_\tau[1/F_z]$ can be identified with $(\varprojlim_{E_0 \subset E} S^\bu(J_E^*)[1/F_z])$
and that Ferrand's Theorem implies that 
$$ U_z \ \equiv \ \Spec (A_\tau[1/F_z])_0 \ =\ \Spec \varprojlim_{E_0 \subset E} (S^\bu(J_E^*)[1/F_z])_0$$ 
is the scheme-theoretic colimit of affine schemes and closed immersions.   We define the sheaf of 
regular functions on $\varinjlim_{E \in \cE(\tau)} \Proj(S^\bu((J_E^j)^*))$
by using the evident identification of the images of  $(A_\tau[1/F_z])_0$ and$ (A_\tau[1/F_{z^\prime}])_0$ in 
$A_\tau[1/F_zF_{z^\prime}])_0$.  So defined, $\varinjlim_{E \in \cE(\tau)} \Proj(S^\bu((J_E)^*))$ is isomorphic
as a local ringed space to $\Proj(A_\tau)$.

To show that $\bP X_\tau \equiv \Proj (A_\tau)$ equipped with the natural maps 
$\bP X_E \to \bP X_\tau, \ E \in \cE(\tau)$ is the colimit
in the category of schemes, we consider a compatible family of maps $\bP X_E \to Z$ for varying $E \in \cE(\tau)$
with $Z$ an arbitrary $k$-scheme.  We must show that this data uniquely  determines a map $\bP X_\tau \to Z$.
Such a map is determined by its restrictions to each open subset $U_z \subset 
\bP X_\tau$; similarly, the restrictions to the pre-images of
 $U_z$ of a compatible family determine the family.  Thus, Proposition \ref{prop:limits} with $A_\tau$ replaced
 by $A_\tau[1/F_z]$ for all $z \in \bP X_\tau$ implies that $\bP X_\tau$ is the scheme-theoretic colimit 
$ \varinjlim_{E \in \cE(\tau)} (\bP X_E)$.
\end{proof}

In the following proposition we shall consider 
\begin{equation}
\label{eqn:bUtau} 
\bU_\tau \ \equiv  \bP X_\tau \backslash \bP X_\tau^{(2)}.
\end{equation}
For the definition of $\bU_\tau$, we observe
that $E^\prime \subset E$ implies that $J_{E^\prime}^2 = J_{E^\prime} \cap J_E^2$ and thus induces a 
closed immersion $\bU_{E^\prime} \to \bU_E$.

\begin{prop}
\label{prop:Utau}

The quasi-projective variety $\bU_\tau$ is the scheme-theoretic colimit $\varinjlim_{E \in \cE(\tau)} \{ \bU_E \}$.
The map $p_\tau: X_\tau \ \to \ Y_\tau$ induces a $\tau$-equivariant map of schemes
\ $p_\tau: \bU_\tau \ \to \  \bP Y_\tau.$

For each elementary abelian $p$-subgroup $E \subset \tau$, a choice of splitting of the map $J_E \to J_E/J_E^2$
determines a splitting $s_E: \bP Y_E \to \bU_E$ which composes with the inclusion $\bU_E \to \bU_\tau$
to induce a map $s_E: \bP Y_E\ \to \ \bU_\tau$ whose composition with $p_\tau$
is the canonical inclusion $\bP Y_E \to \bP Y_\tau$:
\begin{equation}
\label{diag:sec}
\begin{xy}*!C\xybox{%
\xymatrix{
&  \bU_\tau \ar[d]^{p_\tau} \\
\bP Y_E \ar[ur]^{s_E} \ar[r]_{\hookrightarrow} & \bP Y_\tau.
 }
}\end{xy}
\end{equation}
Consequently, $p_\tau$ is surjective.
\end{prop}

\begin{proof}
One readily verifies that both $\bU_\tau$ and $\varinjlim_{E \in \cE(\tau)} \{ \bU_E \}$ represent the same
open subfunctor of $\bP X_\tau$.  

We define the projection $p_E: \bU_E \to \bP Y_E$ to be the projection off the linear subspace
$\bP X_E^{(2)} \subset \bP X_E$.   This is natural with respect to inclusions $E^\prime \to E$.  Thus, we 
obtain the quotient map  $p_\tau: \bU_\tau \ \to \  \bP Y_\tau$ once we identify each $\bP Y_E$ as 
the quotient space of the projection $p_E$.

Writing $J_E = J_E^2 \oplus J_E/J_E^2$ (as $k$-vector spaces) enables us to view 
$$\bP X_E \ = \ \Proj S^\bu((J_E^*) \ \simeq \ \Proj S^\bu((J_E^2)^*) \# \Proj S^\bu((J_E/J_E^2)^*) \ 
= \ \bP X_E^{(2)} \# \bP Y_E, $$
the union of projective lines in $\bP X_E$ from a point of $\bP X_E^{(2)}$ to a 
point of $\bP Y_E$.    The projection $p_E$ sends a point on such a line $\ell $ other than the 
point $\ell \cap \bP X_E^{(2)}$ to the point $\ell \cap \bP Y_E$.  Observe that the intersection of the linear 
subspaces  $\bP X_E^{(2)}$ and  $\bP Y_E$ in $\bP X_E^{(2)} \# \bP Y_E$
is empty.   We define $s_E: \bP Y_E\to \bU_E$
as the inclusion  of $\bP Y_E$ in $\bP X_E^{(2)} \# \bP Y_E$.
\end{proof}

\begin{remark}
\label{remark:iterated}
The proof of Proposition \ref{prop:limits} supplemented by the proof of Proposition \ref{prop:projlimits} shows 
that the colimit of schemes $\bP X_\tau \simeq \varinjlim_{\cE(\tau)} \bP X_E$ in Proposition \ref{prop:projlimits} 
is constructed as iterated 
push-outs of squares of the form 
\begin{equation}
\label{diag:push-outs}
\begin{xy}*!C\xybox{%
\xymatrix{
\bP X_{E^\prime} \ar[d] \ar[r]  & W  \ar[d] \\
\bP X_E  \ar[r] & Z 
 }
}\end{xy}
\end{equation}
where $\bP X_{E^\prime} \to  \bP X_E$ is a closed linear embedding of projectives spaces and 
$\bP X_{E^\prime} \to W$ is a closed immersion (see (\ref{diag:cocart1})).

The colimit $\bP Y_\tau \simeq  \varinjlim_{\cE(\tau)} \bP Y_E$ is constructed as iterated push-out squares of 
the same form with $X_E$ replaced by $Y_E$, and the colimit $\bU_\tau \simeq \varinjlim_{\cE(\tau)} \bU_E$
as iterated pushout squares with $\bP X_{E^\prime} \to \bP X_E $ replaced by the closed immersion
$\bP X_{E^\prime} \backslash  \bP X_{E^\prime}^{(2)} \ \to \bP X_E  \backslash  \bP X_{E}^{(2)}.$
The latter map is closed since the image of $\bP X_{E^\prime} \backslash  \bP X_{E^\prime}^{(2)} $ in 
$\bP X_E  \backslash  \bP X_{E}^{(2)}$ is the intersection of $\bP X_{E^\prime}^{(2)}$ and 
$\bP X_E  \backslash  \bP X_{E}^{(2)}$ inside $\bP X_E$ because $J_{E^\prime}^2$ is the intersection
of $J_{E^\prime}$ and $J_E^2$ inside $J_E$.
\end{remark}

\begin{cor}
\label{cor:surj}
For each elementary abelian $p$-subgroup $E \subset \tau$, a choice of splitting of the map $J_E \to J_E/J_E^2$
determines a splitting $s_E: \bP Y_E \to \bU_E$ which composes with the inclusion $\bU_E \to \bU_\tau$
to induce a map $s_E: \bP Y_E\ \to \ \bU_\tau$ whose composition with $p_\tau$
is the canonical inclusion $\bP Y_E \to \bP Y_\tau$:
\begin{equation}
\label{diag:SE}
\begin{xy}*!C\xybox{%
\xymatrix{
&  \bU_\tau \ar[d]^{p_\tau} \\
\bP Y_E \ar[ur]^{s_E} \ar[r]_{\hookrightarrow} & \bP Y_\tau.
 }
}\end{xy}
\end{equation}
Consequently, $p_\tau$ is surjective.
\end{cor}

\begin{proof}
Given the splitting $J_E = J_E^2 \oplus J_E/J_E^2$ (as $k$-vector spaces), we define
$$s_E: \Proj S^\bu((J_E/J_E^2)^*) \to \ \Proj S^\bu((J_E^*) \ \simeq \ \Proj S^\bu((J_E^2)^*) \# \Proj S^\bu((J_E/J_E^2)^*)$$
by sending $\langle y_0,\ldots,y_N \rangle$ to $\langle 0,\ldots,0,y_0,\ldots,y_N \rangle$.
\end{proof}

We can construct other schemes associated to $\tau$  in a manner strictly analogous to 
the constructions of $X_\tau, \ Y_\tau, \ \bP X_\tau, \ \bP Y_\tau$ and $\bU_\tau$.  We omit the verification
of these analogous constructions as summarized in the next proposition.

\begin{prop}
\label{surj}
We may replace $E \in \cE(\tau) \ \mapsto \ J_E$ in the construction of $X_\tau$ and $\bP X_\tau$
by $E \in \cE(\tau) \ \mapsto \ J_E^j$ for some $j, 1 < j < p$, yielding schemes $X^{(j)}_\tau$ and
$\bP X_\tau^{(j)}$.  In the special case $\tau = E$, for example, $X_E^{(j)} \ = \ \bA_{J_E^j} \ \equiv \ \Spec S^\bu((J_E^j)^*)$.

The quotient maps $S^\bu((J_E^{j})^*) \to S^\bu((J_E^{j+1})^*), \ 1 < j < p-1$ for $E \in \cE(\tau)$ 
induce  closed immersions 
$\bP X^{(j+1)}_\tau \ \subset \ \bP X^{(j)}_\tau $ equivariant under the action
of $\tau$, with open complement 
$$\bU^{(j)}_\tau \ \equiv \ \bP X^{(j)}_\tau \ \backslash \bP X^{(j+1)}_\tau$$ 
mapping naturally to $\Proj (\varprojlim_{E\in \cE(\tau)} S^\bu((J_E^j/J_E^{j+1})^*))$.
\end{prop}

\begin{remark}
\label{remark:filt}
The natural $\tau$-equivariant embeddings \ $\bP X_\tau^{(j)} \ \hookrightarrow \ \bP X_\tau$ \ provide a filtration 
of $\bP X_\tau$ by closed subschemes.

The fibre of $p_\tau: \bU_\tau \to \bP Y_\tau$ above a point $x \in \bP Y_\tau$ can be described as 
follows.  Let $\{ E_1,\ldots,E_s$ \} be the set of elementary abelian subgroups of $\tau$ with the property
that $x\in \bP Y_{E_i} \subset \bP Y_\tau$.  Then $p_\tau^{-1}(x)$ is
the quotient of $ \coprod_{i=1}^s \bA_{J_{E_i}^2}$ by the equivalence relation in which a point of
$\bA_{J_{E_\ell}^2}$ is identified with its image in  $\bA_{J_{E_i}^2}$ whenever $E_\ell \subset E_i$. 
\end{remark}

The following proposition incorporates a fundamental theorem of D. Quillen \cite[7.1]{Q1}.  (The action of
$\tau$ on $\Spec H^*(\tau,k)$ is trivial.)

\begin{prop}
\label{prop:mod-tau}
There are natural (with respect ot $\tau$), $\tau$-equivariant, surjective maps $$X_\tau \ \to \ Y_\tau \ \to \ \Spec H^\bu(\tau,k)$$
equivariant with respect to $\tau$ (where $\tau$ acts trivially on $\Spec H^*(\tau,k)$);
the map $(Y_\tau)/\tau \ \to \ \Spec H^\bu(\tau,k)$ is a $p$-isogeny (denoted $\stackrel{\approx}{\to}$).

These maps induce surjective maps 
$$\bU_\tau   \ \stackrel{p_\tau}{\to} \  \bP Y_\tau \ \to \ (\bP Y_\tau)/\tau \ 
\stackrel{\approx}{\to} \ \Proj H^\bu(\tau,k),$$
contravariantly functorial with respect to inclusions of finite groups.
\end{prop}

\begin{proof}
The map 
$Y_\tau \ \to \ \Spec H^\bu(\tau,k)$ is the colimit with respect to $E \in \cE(\tau)$ of 
maps of schemes induced by the compositions of restriction 
maps $H^*(\tau,k) \to H^*(E,k)$ followed by the $p$-isogenies $H^\bu(E,k) \stackrel{\simeq}{\to} \simeq  
H^\bu(E,k)_{red} \simeq S^\bu((J_E/J_E^2)^*)$.
Let $\tilde \cE(\tau)$ denote the category whose objects are elementary abelian $p$-subgroups of $\tau$ and
whose maps are compositions of inclusions $E^\prime \hookrightarrow E$ with isomorphisms $E \stackrel{\sim}{\to} E^x$
given by conjugation by some $x \in \tau$.  Then Quillen's theorem \cite[7.1]{Q1} tells us that the natural map 
$H^\bu(\tau,k) \to \varprojlim_{\tilde \cE(\tau)} H^\bu(E,k)$ 
is a ``$p$-isogeny" (i.e., has kernel and cokernel whose elements have some $p$-th power 0)
(see \cite[8.10]{Q2}).
We verify by inspection that the natural map  \ $(\varprojlim_{\cE(\tau)} H^\bu(E,k))^\tau \ \to \ 
\varprojlim_{\tilde \cE(\tau)} \Spec H^\bu(E,k)$ is an isomorphism, thereby implying 
the homeomorphism 
\begin{equation}
\label{eqn:quillen}
(Y_\tau)/\tau \ = \ \Spec ((B_\tau)^\tau) \ \stackrel{\approx}{\to} \ \Spec H^\bu(\tau,k).
\end{equation}

To verify the naturality with respect to $\tau$, we observe that a homomorphism $\phi: \sigma \to \tau$ of finite groups
induces a map of partially order sets $\cE(\sigma) \to \cE(\tau)$ and thus a map of commutative $k$-algebras 
$B_\tau \to B_\sigma$:  for each $F \in \cE(\sigma)$, the composition 
$B_\tau \to B_\sigma \to S^*((J_F/J_F^2)^*)$ is given by the projection $A_\tau \to S^*((J_{\phi(F)}/J_{\phi(F)}^2)^*)$ 
followed by the
surjective map $S^*((J_{\phi(F)}/J_{\phi(F)}^2)^*) \to S^*((J_F/J_F^2)^*)$.  This applies equally to determine 
$A_\tau \to A_\sigma$.  A similar argument shows that the embedding $B_\tau \to A_\tau$ determining $X_\tau \to Y_\tau$
is also natural in $\tau$.

The natural maps $H^\bu(\tau,k) \to H^\bu(E,k)$ induced by $E \subset \tau$ in $\cE(\tau)$ are graded, thereby inducing  
 $$\bP Y_\tau \ = \ \varinjlim_{\cE(\tau)} \bP Y_E   \to \ \Proj H^\bu(\tau,k).$$
 This map factors through
$(\bP Y_\tau)/\tau \ \stackrel{\approx}{\to} \ \Proj H^\bu(\tau,k)$, since the action of $\tau$ on $\Proj H^\bu(\tau,k)$
is trivial.    The construction $\tau \mapsto \bP Y_\tau$ is natural with respect to injective group homomorphisms 
$\phi: \sigma \to \tau$ is injective; the injectivity condition iof $\phi$ mplies that $F \to \phi(F)$ is an isomorphism for all $F \in \cE(\sigma)$
so that we may apply $\Proj$ to the graded maps $S^*((J_{\phi(F)}/J_{\phi(F)}^2)^*) \to S^*((J_F/J_F^2)^*)$.
Thus, we obtain embeddings of colimits
$$\varinjlim_{F \in \cE(\sigma)} \bP Y_F \ \to \ \varinjlim_{E \in \cE(\tau)} \bP Y_E, \quad 
\varinjlim_{F \in \cE(\sigma)} \bU_F \ \to \ \varinjlim_{E \in \cE(\tau)} \bU_E.$$
\end{proof}

\vskip .2in


\section{$\tau$-equivariant quasi-coherent sheaves on $X_\tau$ and $\bP X_\tau$}
\label{sec:tau-equiv}

We begin this section with a few generalities about $\tau$-equivariant sheaves on a 
scheme, with $\tau$ an arbitrary discrete group.  Beginning with Definition \ref{defn:Theta-tau},
we return to the context of Section \ref{sec:tau} in which $\tau$ is finite.  In Definition \ref{defn:Theta-tau},
we introduce the operator $\Theta_\tau \in A_\tau \ \otimes\  k\tau$, our analogue of the universal 
$p$-nilpotent operator $\Theta_G \in k[V_r(G)]\otimes kG$ for an infinitesimal group scheme $G$ of
height $\leq r$ (as in \cite{SFB1}).    As defined in Definition \ref{defn:Theta-tauM}, $\Theta_\tau$ determines an 
$A_\tau$-linear map 
$$\Theta_{\tau,M}: A_\tau \otimes M \ \to \ A_\tau \otimes M$$
for any $k\tau$-module $M$;
we verify that $\Theta_{\tau,M}$ is  $p$-nilpotent and $\tau$-equivariant.  Taking kernels, cokernels, and images of 
$\Theta_{\tau,M}^j$ for some $j \ 1 \leq j < p$, 
leads to coherent sheaves on $X_\tau$ (see Theorem \ref{thm:affine-vec}) and on $\bP X_\tau$ 
(see Theorem \ref{thm:projbundle}).  If $M$ has constant $j$-type, then these coherent sheaves associated
to $\Theta_{\tau,M}^j$ are vector bundles (i.e., locally free coherent sheaves).

\begin{defn}
\label{defn:equiv}
Let $X$ be a $k$-scheme and $G$ a group scheme over $k$ with multiplication
$m: G \times G \to G$.  An action of
$G$ on $X$ is the data of an action morphism (over $k$) $\mu: G \times X \to X$ satisfying the 
usual identities required for a group action.  

Let $F$ be a quasi-coherent sheaf on such a $k$-scheme $X$ equipped with a $G$-action over $k$.
The structure of a $G$-equivariant sheaf on $F$ consists of an isomorphism of sheaves on $G\times X$
$$\phi_F: \mu^*F \stackrel{\sim}{\to} p^*F$$ 
where $p: G\times X \to X$ is the projection onto the second factor; $\phi$ is required  to satisfy 
the conditions that its restriction to $\{ e \} \times X$ is the identity on $F$ and that its pull-backs via
$1\times \mu, m \times 1: G\times G \times X \ \to \ G\times X$ are suitably related.  We may
view $\phi_F$ as the data for each ``point" $g$ of the group scheme $G$ an $\cO_X$-linear map
$\phi_{F,g}: \mu_g^*(F) \to F$.

A map $\psi: F \to F^\prime$ of $G$-equivariant, quasi-coherent sheaves on $X$ is a map of quasi-coherent sheaves
which commutes with the $G$ action in the sense that $\phi_{F^\prime} \circ \mu^*(\psi) = p^*(\psi) \circ \phi_F$.
\end{defn}

\vskip .1in
\begin{remark}
A $G$-equivariant, quasi-coherent sheaf on a $k$-scheme $X$ equipped with the action of a group scheme $G$ is
equivalent to a quasi-coherent sheaf on the quotient stack $[X/G]$.
\end{remark}

For a discrete group $\tau$, we can view a $\tau$-equivariant, quasi-coherent sheaf $F$
on $X$ (equipped with an action of $\tau$) as follows.  As a scheme $\tau \times X$ is the disjoint union
$\coprod_{g \in \tau} {}^gX$ where each ${}^gX$ is a copy of $X$.  For each open subset $U \subset X$,
the isomorphism $\phi_F: \mu^*F \stackrel{\sim}{\to} p^*F$ 
restricts to an isomorphism of $\cO_X({}^gU)$-modules $\phi_{g,U}: F({}^gU) \stackrel{\sim}{\to} F(U)$,
where ${}^gU \subset X$ is the open subset of points $x \in X$ such that $gx \in U$.
The conditions of Definition \ref{defn:equiv} are that $\phi_{1,U}$ is the identity and that 
$\phi_{h\cdot g,U} \ = \ \phi_{g,U} \circ  \phi_{h,{}^gU} $.

A map $f: F \to F^\prime$ of $\tau$-equivariant, quasi-coherent sheaves on $X$  is a map of of quasi-coherent 
sheaves determining commutative squares for each open subset $U \subset X$ and each $g \in \tau$:
\begin{equation}
\label{diag:conj}
\begin{xy}*!C\xybox{%
\xymatrix{
F({}^gU) \ar[d]_-{f_{|{}^gU}} \ar[r]^{\phi_{g,U}} & F(U) \ar[d]^-{f_{|U}} \\
F^\prime({}^gU) \ar[r]_{\phi_{g,U}^\prime} & F^\prime(U).
 }
}\end{xy}
\end{equation}
If $\psi: F \to G$ is a map of $\tau$-equivariant, quasi-coherent sheaves on $X$, then
the kernel and cokernel of $\psi$ in the category of quasi-cohoerent sheaves are both $\tau$-equivariant.

\vskip .1in

\begin{ex}
\label{ex:strsheaf}
If $X$ is a scheme equipped with a $\tau$-action, then the structure sheaf $\cO_X$ is naturally $\tau$-equivariant.
For any open $U \subset X$ and any $g \in \tau$, we define $\phi_{g,U}: \cO_X({}^gU) \to \cO_X(U)$ by
sending $f \in \cO_X({}^gU)$ to $g(f) \in \cO_X(U)$, where $g(f)(y) = f(g^{-1}y)$.  

Moreover, if $M$ is a $k\tau$-module, then 
$\cO_X\otimes M$  has the natural structure of a (free) quasi-coherent, $\tau$-equivariant $\cO_X$-module. 
The action of $g \in \tau$ sends the simple tensor $f \otimes m \in (\cO_X\otimes M)({}^gU)$ to $ g(f) \otimes g\cdot m 
\in (\cO_X\otimes M)(U).$  For $h \in \cO_X({}^g(U)$, $h(f\otimes m) = h(f) \otimes m \in (\cO_X\otimes M)({}^gU)$
is sent by the action of $g \in \tau$ to $ g(h(f))\otimes g\cdot m$ which equals the action 
of $\phi_{U,g}(h)$ on $g(f) \otimes g\cdot m$.
\end{ex}

\begin{ex}
\label{ex:tau-affine}
If $X \ = \ \Spec A$ for some finitely generated, commutative $k$-algebra equipped with an action of $\tau$ on $A$
by $k$-algebra automorphisms, then a $\tau$-equivariant, quasi-coherent $\cO_X$-module
is equivalent (by taking global sections on $X$) to an $A$-module $M$ equipped with a group action of $\tau$ 
on $M$ as a $k$-vector space such that \ $g(a\cdot m) \ = \ g(a)\cdot g(m)$.  In other words, $\phi_{g,X}: M \to M$
satisfies the condition that $\phi_{g,X}(a\cdot m) = \phi_{g,X}(a)\cdot \phi_{g,X}(m)$.  A $\tau$-equivariant map
$\psi: M \to M^\prime$ of $\tau$-equivariant $A$-modules is an $A$-module homomorphism satisfying the 
condition that $\psi(\phi_{g,X}(m)) \ = \ \phi_{g,X}(\psi(m))$.

We reformulate such $\tau$-equivariant $A$-modules using the non-commutative $k$-algebra $A \# k\tau$,
the ``smash" or ``semi-direct" product of the Hopf algebra $k\tau$ and the left $k\tau$-module algebra $A$
as in \cite[4.1.2]{M}.  The multiplication in $A \# k\tau$ is given by 
$$(f_1\# x_1)(f_2\# x_2) \ = \ f_1\cdot x_1(f_2) \#x_1x_2, \quad f_1,f_2 \in A, \ x_1,x_2 \in \tau.$$
For any $k\tau$-module $M$, we have an $A \# k\tau$-module structure on $A \otimes M$ 
$$\bu: (A \# k\tau) \otimes (A \otimes M ) \ \to \  A \otimes M,
\quad (f\# x) \bu (h\otimes m) \ = \ f\cdot x(h) \otimes xm.$$
Given this $A\# k\tau$-action on $A\otimes M$, the action of $A$ on $A \otimes M$ is $A$-linear and the action of $\tau$
on $A\otimes M$ is given by $(x\in \tau, \ f\otimes M \in A\otimes M) \ \mapsto \ x(f) \otimes xm$.
\end{ex}

\vskip .1in

In what follows, we shall be especially interested in the special case of Example \ref{ex:tau-affine} in which $A = A_\tau$,
the commutative $k$-algebra introduced in Proposition \ref{prop:limits} for a finite group $\tau$.  We view $A_\tau \otimes M$
as a $\tau$-equivariant $A_\tau$-module as above.  In particular, $A_\tau \otimes k\tau$ is a $\tau$-equivariant
 $A_\tau$-module.

\begin{defn}
\label{defn:Theta-tau}
Consider some $g \in \tau, g^p = e, g \not= e$.  We define 
$(g-e)^\vee \ \in \ \varprojlim_{E \in \cE(\tau)} \ J_E^*$
as follows:  the projection of $(g-e)^\vee$ to $J_E^*$ is the linear function on $J_E$ 
sending the basis element $x-e \in J_E$ to 1 if $x = g$ and sending the basis element $x-e$ 
to 0 if $g \not= x \in E$.

We define
\begin{equation}
\label{eqn:Theta-formula} 
\Theta_\tau \ \equiv \ \sum_{g\in \tau, g^p=e}  (g-e)^\vee \otimes (g-e) \ \in 
\varprojlim_{E \in \cE(\tau)} \ J_E^* \otimes k\tau \ \subset A_\tau \otimes k\tau,
\end{equation}
where $A_\tau \ \equiv \ \varprojlim_{E \in \cE(\tau)} S^\bu(J_E^*)$ is introduced in Proposition \ref{prop:limits}.

Let $\xi: A_\tau \to K$ be a $K$-point of $X_\tau$ lying in $\bA_{J(E)} \subset X_\tau$, thus factoring through an
algebra map $S^\bu(J_E^*) \to K, \ (g-e)^\vee \mapsto a_{g-e}$ for some $p^r-1$-tuple $\{ a_{g-e} \} \in (\bA_{J_E})(K)$ 
(where $r = rank(E)$).   We denote by 
\begin{equation}
\label{eqn:Theta-xi}
\Theta_{\tau,\xi} \ \equiv \  \sum_{g \in E} a_{g-e}\cdot (g-e) \ \in \ KE \ \subset \ K\tau,
\end{equation}
the specialization of $\Theta_\tau$ along $\xi$, and we set
$$\alpha_\xi: K[u]/u^p \ \to \ K\tau, \quad u \mapsto  \Theta_{\tau,\xi}.$$  
\end{defn}

\vskip .1in

\begin{remark}
\label{remark:contrast}
The preceding definition in the special case $\tau = E = \bZ/p^{\times r}$,
$$\Theta_E \ \equiv \ \sum_{g\in E} (g-e)^\vee \otimes (g-e) \ \in \ S^\bu(J_E^*) \otimes kE,$$
does not seem 
consistent with the formulation of $\Theta_{\bG_{a(1)}^{\times r}}$ given in \cite{FP4},
$$\Theta_{\bG_{a(1)}^{\times r}} \ \equiv \ \sum_{i=1}^r X_i \otimes x_i \ \in \
k[V({\bG_{a(1)}^{\times r}} )] \otimes {\bG_{a(1)}^{\times r}}.$$
Let $\{ g_1,\ldots,g_r \}$ be a minimal
set of generators of $E$ and consider the isomorphism
$$k\bG_{a(1)}^{\times r} \ \stackrel{\sim}{\to} \ kE, \quad x_i \mapsto g_i-e$$
where we have identified  $k\bG_{a(1)}^{\times r}$ with $k[x_1,\ldots,x_r]/(x_i^p)$.
Sending $\ol{g_i-e}$ to $g_i-e$ determines the $k$-linear map $s_E: J_E/J_E^2 \to J_E$ 
and thus a map of affine varieties $s_E: \bA_{J_E/J_E^2} \to \bA_{J_E}$.
Restricting $\Theta_E \in S^\bu(J_E^*) \otimes kE$ along $s_E :
S^\bu(J_E^*) \to S^\bu((J_E/J_E^2)^*)$, we obtain the operator corresponding to 
$\Theta_{\bG_{a(1)}^{\times r}}$ under this equivalence of group algebras.
\end{remark}

\vskip .1in

\begin{defn}
\label{defn:Theta-tauM}
For any $k\tau$-module $M$, we define the  $A_\tau$-linear map
$$\Theta_{\tau,M}: A_\tau \otimes M   \ \to \ A_\tau \otimes M  \quad 1 \otimes m  \mapsto  \sum_{g \in \tau, g^p=e} 
(g-e)^\vee \otimes (gm-m).$$.
\end{defn}

\begin{prop}
\label{prop:Theta-tauM}
For any $k\tau$-module $M$, $\Theta_{\tau,M}: A_\tau \otimes M   \ \to \ A \otimes M$ is $\tau$-equivariant; in other
words, $\Theta_{\tau,M}$ is an endomorphism of $A_\tau \otimes M$ as an $A \# k\tau$-module. 

Moreover, the restriction of $\Theta_{\tau,M}$ along some geometric point $\xi: A_\tau \to K$, 
$$\Theta_{\tau,M,\xi}: K\otimes M \ \to  \ K \otimes M, \quad  m \mapsto \sum_{g \in E} a_{g-e} \otimes (g-e)(m),$$
is given by the action of  $\Theta_{\tau,\xi}  \in K\tau$  of (\ref{eqn:Theta-xi}) on the $K\tau$-module $K\otimes M$.
\end{prop}

\begin{proof}
Let $h \in \tau$.  We verify that $h \circ (g-e)^\vee \ = \ (g^h-e)^\vee$ by comparing values on $g^\prime -e$
for all $g^\prime \in \tau$.  Thus, applying $h$ to $(\Theta_{\tau,M})(f \otimes m)$ gives
$$h \circ (\sum (g-e)^\vee f \otimes (gm-m)) \ = \  \sum(g^h-e)^\vee \cdot h(f) \otimes (hg(m)-hm).$$
On the other hand, applying $(\Theta_{\tau,M})$ to $h(f \otimes m) = h(f) \otimes hm $ gives \\
$\sum (g-e)^\vee \cdot h(f) \otimes (gh(m) - hm)$ which equals the previous expression once we re-index
the sum by $g^h$ rather than $g$.

The description of $\Theta_{\tau,M,\xi}$ follows from  the explicit description of $(g-e)^\vee$ and
the fact that extension along
$\xi: A_\tau \to K$ is achieved by evaluating elements of $A_\tau$.
\end{proof}

\begin{remark}
It is reasonable to consider the Jordan type of $\alpha_\xi^*(A_\tau \otimes M) \in K[u]/u^p$ to be the local Jordan
type of the $\tau$-module $M$ at $\xi \in X_\tau$.  As observed in Proposition \ref{prop:Theta-tauM}, 
this is the Jordan type of the endomorphism $\Theta_{\tau,M,\xi}$.  Upon identifying a $k$-rational point of $X_\tau$
with an element $u$ of some $J_E$, we can identify the local Jordan type of $M$ at $u$ with the Jordan
type of $u \in k\tau$ as an endomorphism of the $k\tau$-module  $M$.
\end{remark}
\vskip .1in

The naturality with respect to $\tau$ of $\Theta_{\tau,M}$ is described in the next proposition.

\begin{prop}
\label{prop:tau-natural}
Let $\phi: \sigma \to \tau$ be a homomorphism of finite groups and consider a $k\tau$-module $M$ together
with its restriction to $k\sigma$.  Then  $\Theta_{\sigma,M}^j$ is the extension along $A_\tau \to A_\sigma$
of $\Theta_{\tau,M}^j$.

Consequently, there are natural $\sigma$-equivariant homomorphisms of  $A_\sigma$-modules
$$ker\{ (\Theta_{\tau,M})^j\} \ \to \ ker\{ (\Theta_{\sigma,M})^j\} $$
and similar homomorphisms with kernel replaced by cokernel or image.
\end{prop}

\begin{proof}
Let $\phi^*: A_\tau \to A_\sigma$ denote the map of $k$-algebras induced by $\phi$.  For $g \in \tau$ not in 
the image of $\phi$,  $\phi^*((g-e)^\vee) = 0$; on the other hand, if $g = \phi(h)$ with $h^p = e$, then 
$\phi^*((g-e)^\vee) = (h-e)^\vee$.  Thus, if $M$ is a $\tau$-module and $m \in M$, then
$$\Theta_{\sigma,M}(m)  \ = \ \sum_{h\in \sigma,h^p=e} (h-e)^\vee \otimes (h\cdot m-m)$$
is the image of $\Theta_{\tau,M}(m)$ under the base change map $\phi^*\otimes 1: A_\tau \otimes M \to 
A_\sigma \otimes M$ as stated in the first assertion of the proposition for $j = 1$.   For $j > 1$, we 
iterate this argument.

The naturality of $A_\sigma \otimes_{A_\tau} (-)$  determines the commutative diagram whose upper row
(respectively, lower row) is an exact sequence of $\tau$-equivariant $A_\tau$ (resp., $\sigma$ equivariant
$A_\sigma$) modules
\begin{equation}
\label{diag:conjM}
\begin{xy}*!C\xybox{%
\xymatrix{0 \ar[r] & ker\{ (\Theta_{\tau,M})^j\} \ar[r]  & A_\tau \otimes M \ar[r]^{(\Theta_{\tau,M})^j} \ar[d] 
& A\otimes M \ar[d] \\
0 \ar[r] & ker\{ (\Theta_{\sigma,M})^j\} \ar[r]  & A_\sigma \otimes M \ar[r]^{(\Theta_{\sigma,M})^j} & A_\sigma \otimes M
 }
}\end{xy}
\end{equation}
and thus determines a map of  $A_\sigma$-modules $ker\{ (\Theta_{\tau,M})^j\} \ 
\to \ ker\{ (\Theta_{\sigma,M})^j\} $.  The fact that $\phi^*: A_\tau \to A_\sigma$ is 
$\sigma$-equivariant implies that this map is a $\sigma$-equivariant map of $A_\sigma$-modules.

A similar argument applies with kernel replaced by cokernel or image.
\end{proof}

We shall utilize the following useful criterion for a coherent sheaf $\cF$ on a reduced Noetherian scheme $X$ to be 
locally free (i.e., a vector bundle over $X$); this is stated as an exercise in \cite[II.5.ex5.8]{Har} and a 
proof is given in \cite[4.11]{FP4}.

\begin{prop}
\label{prop:locfree}
Let $X$ be a reduced Noetherian scheme and $\cF$ a coherent sheaf on $X$.  For each point $x \in X$,
let $\cO_{X,x}$ denote the stalk at $x$ of the structure sheaf $\cO_X$, let $\cF_{(x)}$ denote stalk at $x$ of $\cF$,
and let $k(x)$ denote the residue field of $\cO_{X,x}$.Then $\cF$ is locally free
if and only if 
$$dim_{k(x)} (\cF_{(x)}\otimes_{\cO_{X,x}} k(x)) \ = \ dim_{k(y)} (\cF_{(y)}\otimes_{\cO_{X,y}} k(y)) $$
whenever $x, \ y$ lie in the same connected component of $X$.
\end{prop}

 Our first application of Proposition \ref{prop:locfree} is the following proof that $\Theta_{\tau,M}$ is $p$-nilpotent.
 
\begin{prop}
\label{prop:pnilp-affine}
For any finite dimensional $k\tau$-module $M$, the endomorphism $\Theta_{\tau,M}: A_\tau \otimes M \to A_\tau\otimes M$ 
of Definition \ref{defn:Theta-tauM} is $p$-nilpotent.
\end{prop}

\begin{proof}
We consider the short exact sequence
\begin{equation}
\label{eqn:shortexact}
0 \to im\{\Theta_\tau^p \} \to \ A_\tau \otimes M \ \to \ coker\{\Theta_\tau^p \} \to 0
\end{equation}
and its specializations
\begin{equation}
\label{eqn:shortexact-x}
im\{\Theta_{\tau,x}^p \} \to \ k(x) \otimes M \ \to \ coker\{\Theta_{\tau,x}^p \} \to 0
\end{equation}
at points $x \in X_\tau$.  
Since $\Theta_{\tau,x}^p = 0$ for all $x \in X_\tau$,  (\ref{eqn:shortexact-x}) implies that the dimension
of $coker\{\Theta_{\tau,x}^p \}$ as a $k(x)$-vector space equals $dim_k(M)$.  By Proposition \ref{prop:locfree},
$coker\{\Theta_\tau^p \}$ is locally free; thus, (\ref{eqn:shortexact}) is locally split,    On the other hand, specialization
at generic points of $X_\tau$ is exact, so that the vanishing of $\Theta_{\tau,x}^p$ at generic points implies that
the locally split $A_\tau$-submodule $im\{\Theta_{\tau,x}^p \}$ of the free $A_\tau$-module $A_\tau \otimes M$ must be 0.
Thus, $\Theta_{\tau,x}^p = 0.$
\end{proof}

We recall that $K\tau$ is a (left) flat $KE$-module for any $E \in \cE(\tau)$ and that $\alpha_\xi: K[u]/u^p \to KE$
is flat if and only if $\alpha_\xi(u) \in J_E \backslash J_E^2$ which is the case if and only if $\xi$ is a point of 
$X_\tau \backslash X^{(2)}_\tau$

\begin{defn}
\label{defn:j-rank}
A $k\tau$ module $M$ is said to be of  {\bf constant $j$-rank} if the rank of $\alpha^j: M\otimes K \to M\otimes K$ 
is independent of choice of a flat map of $K$-algebras of the form $\alpha: K[u]/u^p \to KE \subset K\tau$ with
$E \in \cE(\tau)$ and field extension $K/k$. 

A $k\tau$ module $M$ is said to be of {\bf constant Jordan type} if the Jordan type $\alpha_K^*M$ as a $K[u]/u^p$-module
is independent of flat map of $K$-algebras of the form $\alpha: K[u]/u^p \to KE \subset K\tau$ with
$E \in \cE(\tau)$ and field extension $K/k$.   A $k\tau$ module $M$ is of constant Jordan type if and only if it
is of constant $j$-rank for each $j, \  1 \leq j < p$.
\end{defn}

\vskip .1in

\begin{thm}
\label{thm:affine-vec}
If $M$ is a finite dimensional $\tau$-module of constant $j$-rank, then the restrictions 
to $X_\tau \backslash X^{(2)}_\tau$ of the $\tau$-equivariant, coherent sheaves on $X_\tau$
$$ker\{ (\Theta_{\tau,M})^j\}, \quad coker\{ (\Theta_{\tau,M})^j) \}, \quad im\{ (\Theta_{\tau,M})^j\} $$
 are $\tau$-equivariant vector bundles.

Consequently, if $M$ is a finite dimensional $\tau$-module of constant Jordan type, then these
coherent sheaves on $X_\tau \backslash X^{(2)}_\tau$ are $\tau$-equivariant vector bundles for any $j, \ 1 \leq j < p$.
\end{thm}

\begin{proof}
The fact that the kernel, cokernel, and image of $(\Theta_{\tau,M})^j$ are $\tau$-equivariant, coherent sheaves on $X_\tau$
arises from the fact that $\tau$-equivariant, coherent sheaves on $X_\tau$ (i.e., $A_\tau \# k\tau$-modules) form an abelian category and that $(\Theta_{\tau,M})^j$ is $\tau$-equivariant.

To show that $ker\{ (\Theta_{\tau,M})^j\}, \ coker\{ (\Theta_{\tau,M})^j) \}, \ im\{ (\Theta_{\tau,M})^j\} $ restricted to
$X_\tau \backslash X^{(2)}_\tau$ are vector bundles if $M$ has constant $j$-rank, it suffices by 
Proposition \ref{prop:locfree} to show that the fiber $\ker\{ (\Theta_{\tau,M})^j\}\otimes_{A_\tau} k(x)$ 
(respectively, $coker\{ (\Theta_{\tau,M})^j) \otimes_{A_\tau} k(x) \}$; \ resp. 
$ im\{ (\Theta_{\tau,M})^j\} \otimes_{A_\tau} k(x)$) has dimension independent of $x  \in X_\tau \backslash X^{(2)}_\tau$.
By definition of constant $j$-rank, the dimension of $im\{ (\Theta_{\tau,M})^j\} \otimes_{A_\tau} k(x)$ is 
independent of $x \in  X_\tau \backslash X^{(2)}_\tau$.

Consider the short exact sequences of sheaves:
$$0 \to ker\{ (\Theta_{\tau,M})^j\}  \ \to \ \cO_{X_\tau} \otimes M \ \to \ im\{ (\Theta_{\tau,M})^j) \} \to 0$$
$$ 0 \to im\{ (\Theta_{\tau,M})^j\} \ \to \  \cO_{X_\tau} \otimes M\ \to \ coker\{ (\Theta_{\tau,M})^j) \} \to 0.$$
Restricted to some neighborhood of $x\in X_\tau\backslash X^{(2)}_\tau$, $ im\{ (\Theta_{\tau,M})^j\}$ is free
so that the upper sequence splits upon restriction to this neighborhood; thus,  $ker\{ (\Theta_{\tau,M})^j\}$
is also locally free upon restriction to $X_\tau \backslash X_\tau^{(2)}$.  Observe that the kernel and cokernel
of a linear endomorphism of a finite dimensional vector space over a field have the same dimension.  Thus, 
applying Proposition \ref{prop:locfree} once again we conclude that the restriction to $X_\tau \backslash X_\tau^{(2)}$ of 
$coker\{ (\Theta_{\tau,M})^j)$ is also locally free.
\end{proof}

We proceed to formulate the projective version of $\Theta_{\tau}$; as defined, $\bP \Theta_\tau$ is
a $\tau$-invariant global section of  the coherent sheaf $\cO_{\bP X_\tau}(1) \otimes k\tau$, the 
first Serre twist of the free $\cO_{\bP X_\tau}$-module
(of rank equal to the order of $\tau$) on $\bP X_\tau$.  We recall that $A_\tau$ is graded, and that 
the vector space of global sections $\Gamma(\bP X_\tau,\cO_{\bP X_\tau}(n))$ of $\cO_{\bP X_\tau}(n)$
equals $(A_\tau)_n$, the summand of $A_\tau$ of elements homogeneous of degree $n$.

In what follows, we implicitly use a theorem of J.-P Serre
which enables us to replace  a coherent sheaf $\cF$ on $\bP X_\tau$ by the graded $A_\tau$-module 
$\bigoplus_{n \geq 0} \Gamma(X_\tau,\cF(n))$
(see \cite[5.15]{Har}) since $A_\tau$ is finitely generated, graded $k$-algebra generated by elements of degree 1.

\begin{defn}
\label{defn:bPTheta}
We define 
$$\bP \Theta_\tau \ \equiv \ \sum_{g\in \tau, g^p=e} (g-e) \otimes (g-e)^\vee \ \in 
\  (A_\tau)_1 \otimes k\tau \ = \ \Gamma(\bP X_\tau,\cO_{\bP X_\tau}(1) \otimes k\tau).$$
For any $k\tau$-module $M$, we define
$$\bP \Theta_{\tau,M}: \cO_{\cO_{\bP X_\tau}} \otimes M \ \to \cO_{\bP X_\tau}(1) \otimes M$$
as in Definition \ref{defn:Theta-tau}, now viewed as a map from the 
graded $A_\tau \otimes k\tau$-module $A_\tau\otimes M$ to the graded $k\tau \otimes A_\tau$-module
$(A_\tau\otimes M)[-1]$.
\end{defn}

\vskip .1in

The following theorem is the projective version of Theorem \ref{thm:affine-vec}. 
 
\begin{thm}
\label{thm:projbundle}
Let $M$ be a finite dimensional $k\tau$-module.  Then $\bP \Theta_{\tau,M}$ of Definition \ref{defn:bPTheta}
is a map of $\tau$-equivariant, coherent sheaves on $\bP X_\tau$ whose $p$-th iteration,
$(\bP \Theta_{\tau,M})^p: \cO_{\cO_{\bP X_\tau}} \otimes M \ \to \  \cO_{\bP X_\tau}(p) \otimes M$, is 0.

For any $j, 1 \leq j < p$, the kernel $ker\{ (\bP \Theta_{\tau,M})^j\}  \subset \cO_{\cO_{\bP X_\tau}} \otimes M$, the
cokernel $ \cO_{\cO_{\bP X_\tau}}(j) \otimes M \  \twoheadrightarrow  coker\{ (\bP \Theta_{\tau,M})^j) \}$, and 
the image  $im\{ (\bP \Theta_{\tau,M})^j\}  \subset \cO_{\cO_{\bP X_\tau}}(j) \otimes M $ are 
 $\tau$-equivariant, coherent sheaves on $\bP X_\tau$.
 
If $M$ is a finite dimensional $\tau$-module of constant $j$-rank, then the restrictions
to $\bU_\tau \equiv \bP X_\tau \backslash \bP X^{(2)}_\tau \ \subset  \bP X_\tau$  of 
$$ker\{ (\bP \Theta_{\tau,M})^j\}, \quad coker\{ (\bP \Theta_{\tau,M})^j) \}, \quad im\{ (\bP \Theta_{\tau,M})^j\} $$
are $\tau$-equivariant vector bundles.

Moreover, if $\phi: \sigma \to \tau$ is an inclusion of finite groups, then 
$$\Phi^*(ker\{ (\bP \Theta_{\tau,M})^j\}) \ = \ ker\{ (\bP \Theta_{\sigma,M})^j\}$$ 
(and similarly with $ker\{- \}$ replaced by $coker\{- \}$ or $im\{- \}$) for any finite dimensional
$\tau$-module $M$, where $\Phi: \bU_\sigma \ \to \ \bU_\tau$ is the closed immersion induced by $\phi$.
\end{thm}

\begin{proof}
The $p$-nilpotence of $\bP \Theta_{\tau,M}$ follows from the fact that $\Theta_M^p: A_\tau \otimes M 
\to A_\tau \otimes M $ is 0 as shown in Proposition \ref{prop:pnilp-affine}.  The fact that the kernel,
cokernel, and image of $(\bP \Theta_{\tau,M})^j$ are $\tau$-equivariant, coherent sheaves on $X_\tau$
arises from the fact that $\tau$-equivariant, coherent sheaves on $\bP X_\tau$ form an abelian category 
and that $(\bP \Theta_{\tau,M})^j$ is $\tau$-equivariant.

Assume that $M$ has constant $j$-rank.  As in the proof of Proposition \ref{prop:projlimits}, let $U_x 
\equiv Spec (A_\tau[1/F_\xi])_0$ be an affine open neighborhood of some point $\xi \in \bU_\tau$.    We essentially
repeat the proof of Theorem \ref{thm:affine-vec} by first observing that $im\{ (\bP \Theta_{\tau,M})^j) \}$ is locally free
restricted to $Spec (A_\tau[1/F_\xi])_0$ by definition of constant $j$-rank and Proposition \ref{prop:locfree}.
We  then use  the short exact sequences of sheaves on $\bP X_\tau$:
$$0 \to ker\{ (\bP \Theta_{\tau,M})^j\}  \ \to \ \cO_{\bU_\tau} \otimes M \ \to \ im\{ (\bP \Theta_{\tau,M})^j) \} \to 0$$
$$ 0 \to im\{ (\bP \Theta_{\tau,M})^j\} \ \to \  \cO_{\bU_\tau}(j) \otimes M\ \to \ coker\{ (\bP \Theta_{\tau,M})^j) \} \to 0$$
to conclude that both $ker\{ (\bP \Theta_{\tau,M})^j\}$ and $coker\{ (\bP \Theta_{\tau,M})^j) \}$ are locally free
when restricted to $\bU_\tau$.

The naturality statement with respect to an inclusion $\phi: \sigma \to \tau$ of finite groups follows from 
Proposition \ref{prop:tau-natural} granted that the inclusion $\phi$ determines $\Phi: \bU_\sigma \ \to \ \bU_\tau$.
\end{proof}

\begin{remark}
\label{remark-BP}
Following \cite{BP}, we should also consider subquotients of the form 
$$ker\{ \bP \Theta_{\tau,M}^{j+1}\} \cap im\{  \bP \Theta_{\tau,M}^{i-j-1} \} /
ker\{ \bP \Theta_{\tau,M}^{j+1}\} \cap im\{  \bP \Theta_{\tau,M}^{i-j} \} + 
ker\{ \bP \Theta_{\tau,M}^j \} \cap im\{  \bP \Theta_{\tau,M}^{i-j-1} \} $$
to obtain further vector bundles associated to $k\tau$-modules of constant Jordan type.
\end{remark}

\vskip .2in


\section{Example of $\tau = E$, an elementary abelian $p$-group}

We recall from \cite[7.5]{FP2} that $\Proj H^\bu(\tau,k)$ is isomorphic to the scheme \ $\Pi(\tau)$ \ of 
equivalence classes of ``$\pi$-points of $\tau$", where a $\pi$-point of $\tau$ is a flat map $\alpha_K: K[u]/u^p \to 
K\tau$ of $K$-algebras factoring through an abelian subgroup for some field extension $K/k$; each equivalence class
of $\pi$-point is represented by a flat map factoring through the group algebra of some $E \subset \tau$.
As mentioned earlier, for an elementary abelian $p$-group $E$ and a field extension $K/k$, 
a $K$-algebra map  $K[u]/u^p \to KE$ is 
flat if and only if $u$ is sent to an element of $J_E \backslash J_E^2$.

The map $\bP Y_\tau \ = \ \varinjlim_{E \in \cE(\tau)} \bP Y_E\ \to \ \Proj H^\bu(\tau,k)$ is induced by the 
maps $\bP Y_E \to \Proj H^\bu(\tau,k)$ for $E \in \cE(\tau)$ which are defined as follows.
An equivalence class of $\pi$-points represented by a non-zero element of $u \in (J_E/J_E^2)\otimes K$  
is sent to the intersection with $H^\bu(\tau,k)$ of the kernel of the map in cohomology induced by 
$\alpha_K: K[t]/t^p \ \to KE, t \mapsto u$ (see \cite{FP2}).  These maps are compatible with
inclusions $E^\prime \to E$.  
Since every $\pi$-point of $k\tau$ is  equivalent to one of this form, the map  
$\bP Y_\tau \ \to \ \Proj H^\bu(\tau,k)$ is surjective.  Since $\tau$ acts trivially on
$H^\bu(\tau,k)$, this map factors as 
$$\bP Y_\tau \ \twoheadrightarrow (\bP Y_\tau )/\tau \ \to  \Proj H^\bu(\tau,k).$$

The following proposition makes explicit the map $S^\bu(J_E^*) \ \to \ S^\bu((J_E/J_E^2)^*)$ induced by the section
$s_E: J_E/J_E^2 \to J_E$ associated to a choice of basis for $E$.

\begin{prop}
\label{prop:radical-map}
Let $E \simeq \bZ/p^{\times r}$ be an elementary abelian $p$-group with basis $g_1,\ldots,g_r \in E$.  
Thus, the natural quotient map $J_E \ \to \ J_E/J_E^2$ sends $g_1^{c_1} \cdots g_r^{c_r}-e \in J_E$ to 
$\sum_{i=1}^r c_i (\ol{g_i-e}) \in J_E/J_E^2$.

The choice $\{ g_1,\ldots,g_r \}$ of basis for $E$ determines the section 
$$s_E: J_E/J_E^2 \  \to \ J_E,  \quad \ol{g_i-e} \mapsto g_i-e .$$  
The induced restriction (quotient) map  $S^*(J_E^*) \to S^\bu((J_E/J_E^2)^*)$ sends 
$(g-e)^\vee$ to $\ol{(g_i-e)}^\vee$ if $g \ = \ g_i$ for some $i$ and to 0 otherwise.
\end{prop}

\begin{proof}
The image of 
$$(g_i\cdot g_1^{c_1} \cdots g_r^{c_r}-e) - (g_i-e) - (g_1^{c_1} \cdots g_r^{c_r}-e) \ = \ (g_i-e)(g_1^{c_1} \cdots g_r^{c_r}-e)$$
in $J_E/J_E^2$ is 0,  so that the formula for the image of $g_1^{c_1} \cdots g_r^{c_r}-e$ follows by induction
on $c = \sum_i c_i \geq 1, \ c_i \geq 0.$

The identification of the restriction map  $S^*(J_E^*) \to S^\bu((J_E/J_E^2)^*)$ induced by $s_E$  is quickly 
checked by observing that the value the dual vector $(h-e)^\wedge$ on the basis  $\{ g-e; e\not= g \in E \}$ is 
equal to 1 if $h = g$ and 0 if $h \not= g$ (see Remark \ref{remark:contrast}).
\end{proof}

One can interpret the next proposition as saying that the construction of coherent sheaves in 
Theorem \ref{thm:projbundle} refines the construction introduced in \cite{FP4} for 
infinitesimal group schemes in the special case of $\bG_{a(1)}^{\times r}$ (whose group algebra
is isomorphic to $k\bZ/p^{\times r})$.  The reader is referred to \cite{BP} and \cite{Ben}, where many 
examples of vector bundles on projective spaces $\bP^{r-1}$ are investigated arising from modules of 
constant Jordan type for elementary abelian $p$-groups $E = \bZ/p^{\times r}$.  

\begin{notation}
We simplify the notation of \cite{FP4}.  For any infinitesimal group scheme $G$ of height $\leq r$,  any
$G$-module $M$, and any integer $j, 1 \leq j < p$, we 
use $(\bP \Theta_{G,M})^j$ in place of the notation $((\tilde \Theta_G)^j,M\otimes \cO_{\Proj k[V_r(G)]})$
of \cite[defn 4.6]{FP4}.
\end{notation}

\begin{prop}
\label{prop:elem}
let $E \simeq \bZ/p^{\times r}$ be an elementary abelian $p$-subgroup 
of $\tau$ and choose elements $g_1,\ldots,g_r \in E$ which generate $E$.   Let $s_E: \bP Y_E \to \bU_E$
be the splitting of $p_E: \bU_E \to \bP Y_E$ given by this basis  (see \ref{diag:sec}).
Then for any finite dimensional $kE$-module $M$,  the pull-backs along $s_E$ of the coherent sheaves 
$$ker\{ (\bP \Theta_{E,M})^j\}, \quad coker\{ (\bP \Theta_{E,M})^j) \}, \quad im\{ (\bP \Theta_{E,M})^j\} $$
on $\bU_\tau$  are naturally identified with the corresponding coherent sheaves 
$$ ker\{(\bP \Theta_{\bG_{a(1)}^{\times r},M})^j \}, \quad  
coker\{(\bP \Theta_{\bG_{a(1)}^{\times r},M})^j  \}, \quad 
\quad \quad im\{(\bP \Theta_{\bG_{a(1)}^{\times r},M})^j \}$$
on $\Proj k[V(\bG_{a(1)}^{\times r})]$  constructed in \cite{FP4}, where $M$ is given the structure of a 
$\bG_{a(1)}^{\times r}$-module using the isomorphism
\begin{equation}
\label{eqn:isoE}
kE \ \stackrel{\sim}{\to} \ k\bG_{a(1)}^{\times r} \equiv k[x_1,\ldots_r]/(x_i^p), \quad g_i \mapsto x_i+1
\end{equation}
\end{prop}

\begin{proof}
We identify $Y_E$ with $\Spec H^\bu(E,k)$.  Consider
$$s_E^*(\bP \Theta_E) \ \in \ \Gamma(\bP Y_E,\cO_{\bP Y_E}(1)\otimes kE),$$
where $\bP \Theta_E \ \in \ \Gamma(\bP X_E,\cO_{\bP X_E}(1)\otimes kE)$ is given in Definition \ref{defn:bPTheta}.
We employ the isomorphism (\ref{eqn:isoE}) inducing a $p$-isogeny 
$$\bP Y_E  \ \stackrel{\simeq}{\to}  \ \Proj H^\bu(\bG_{a(1)}^{\times r},k) \equiv \Proj S^\bu(\fg_a^{*\oplus r}).$$
This isomorphism sends $s_E^*(\bP \Theta_E)$ to 
$$\bP \Theta_{\bG_{a(1)}^{\times r}} \ \in \Gamma(\cO_{\Proj S^\bu(\fg_a^{*\oplus r})}(1) \otimes k\bG_{a(1)}^r) $$
as formulated in \cite[Defn 4.6]{FP4}.

Utilizing the isomorphism (\ref{eqn:isoE}) to identify modules  for $kE$ and for 
$k\bG_{a(1)}^{\times r}$, we easily verify that the isomorphism (\ref{eqn:isoE}) leads to an identification
of $$s_E^*(\bP \Theta_{E,M}): M \otimes \cO_{\bP X_E} \to M \otimes \cO_{\bP X_E}(1)$$
with 
$$\bP \Theta_{\bG_{a(1)}^{\times r},M}: M \otimes \cO_{\Proj S^\bu(\fg_a^{*\oplus r})} \to 
M \otimes \cO_{\Proj S^\bu(\fg_a^{*\oplus r})}(1).$$
In particular, we conclude the asserted identification of coherent sheaves on $\bP Y_E$ and 
$\Proj S^\bu(\fg_a^{*\oplus r})$.
\end{proof}

\begin{remark}
One might consider conditions on a finite dimensional $k\tau$-module $M$ which are analogues of the condition
that $M$ have constant $j$-rank for some $j, 1 \leq j < p$.  For example, one might require $M$ to have the property
that the rank of $u$ as an endomorphism of $M$ is independent of the choice of $u \in J_E^j \backslash J_E^{j+1}$
and $E \in \cE(\tau)$, a seemingly stronger condition on $M$ than having constant $j$-rank.  
This condition should imply that $im\{ \bP_{\tau,M} \}$ restricted to $\bU_\tau^{(j)} \equiv \bP X_\tau^{(j)} \backslash 
\bP X_\tau^{(j+1)}$ is a vector bundle. 
\end{remark}

We understand the algebraic K-theory of projective spaces, but there are many difficult questions 
which remain unanswered about existence and properties of vector bundles on projective spaces (as 
opposed to their stable equivalence classes).  The remainder of this section reflects this lack of understanding.

\begin{question}
\label{ques:iso1}
Consider $E \simeq \bZ/p^{\times r}$ and a $kE$-module $M$ of constant $j$-rank.   Are the isomorphism 
classes of the vector bundles 
$$s_E^*(ker\{ (\bP \Theta_{E,M})^j\}), \quad s_E^*(coker\{ (\bP \Theta_{E,M})^j) \}), \quad s_E^*(im\{ (\bP \Theta_{E,M})^j\}) $$
(see Proposition \ref{prop:elem}) independent of the choice of generators
$\{ g_1,\ldots,g_r \}$ for $E$?
\end{question}

Since the construction of the coherent sheaves on $\bU_E$
$$ker\{ (\bP \Theta_{E,M})^j\}, \quad coker\{ (\bP \Theta_{E,M})^j) \}, \quad im\{ (\bP \Theta_{E,M})^j\} $$
is independent of choice of basis for $E$, Question \ref{ques:iso1} can be equivalently formulated in terms
of the dependence on the choice of section $s_E(-)$ of the isomorphism class of pull-backs of vector bundles on
$\bU_E$.

We provide a related question whose answer escapes us.

\begin{question}
\label{ques:iso2}
Fix a given section $s_E: \bP Y_E \ \to\bU_E$ of $p_E: \bU_E \ \to \ \bP Y_E$ and consider
two $kE$-modules $M$ and $M^\prime$ of constant $j$-rank.   If 
$$s_E^*(\{ ker\{ (\bP \Theta_{E,M})^j\}) \ \simeq \ s_E^*(\{ ker\{ (\bP \Theta_{E,M^\prime})^j\})$$
as vector bundles on $\bP Y_E$, then are 
$ ker\{ \bP (\Theta_{E,M})^j\}, \ ker\{ \bP (\Theta_{E,M^\prime})^j\}$ isomorphic as vector bundles on $\bU_E$? 

Of course, one can ask the same question for cokernel and image vector bundles.
\end{question}

As shown in Proposition \ref{prop:iso-KE}, these question have an affirmative answer if we replace 
isomorphism class of vector bundle by stable equivalence class (i.e., class in the Grotherdieck group $K_0$).

\begin{prop}
\label{prop:aff-fib}
Consider $p_E: \bU_E = \bP X_E \backslash \bP X^2_E \ \to \ Y_E$ as in Proposition \ref{prop:mod-tau} for
some elementary abelian $p$-group $E  \simeq \bZ/p^{\times r}$. 
\begin{enumerate}
\item
 $p_E: \bU_E \to \bP Y_E$ is the map of schemes $Vec(\cO_{\bP Y_E}(-1))^{\oplus N}) \to \bP Y_E$
associated to the vector bundle $(\cO_{\bP Y_E}(-1))^{\oplus N}$, where $N = dim(J_E^2) = p^r-1-r$.
\item
Let $s_E: \bP Y_E \to \bU_E$ be a section of $p_E: \bU_E \to \bP Y_E$ as in Corollary \ref{cor:surj}.  There
exists a map $\Psi: \bU_E \times \bA^1 \to \bU_E$ over $\bP Y_E$  such that 
$\Psi_{\bU_E \times \{ 1 \} } = id_{\bU_E}$ and $\Psi_{\bU_E \times \{ 0 \}} = s_E \circ p_E$.
\item
Let $s_E, s_E^\prime: J_E/J_E^2 \to J_E$ be two sections of the quotient $J_E \twoheadrightarrow J_E/J_E^2$
with associated sections $s_E, s_E^\prime: \bP Y_E \to \bU_E$.   Then there is a map $S_E: \bP Y_E \times \bA_1
\to \bU_\tau$ over $\bP Y_E$ whose restriction to $\bP Y_E \times  \{ 0 \}$ is $s_E$ and whose restriction to 
$\bP Y_E \times  \{ 1 \}$ is $s_E^\prime$.
\end{enumerate}.
\end{prop}

\begin{proof}
Using the isomorphism $\bP X_E \simeq \bP X_E^{(2)} \# \bP Y_E$ determined by a splitting of 
$0 \to J_E^2 \to J_E \to J_E/J_E^2 \to 0$, we readily observe that $p_E: \bU_E \to \bP Y_E$ 
is the structure map (as a map of schemes) of a direct sum of line bundles; we easily check that 
each of these line bundles is isomorphic to $\cO_{\bP Y_E}(-1)$.

We define $\Psi: \bU_E \times \bA^1 \to \bU_E$ over $\bP Y_E$ as follows.  Choose linear
coordinate functions $Y_1,\ldots,Y_N$ spanning the global sections of $\cO_{\bP X_E^{(2)}}(1)$
and linear coordinate functions $X_1,\ldots X_r $ spanning the global sections of $\cO_{\bP Y_E}(1)$.  Over
$V_i =  \bP Y_E - Z(X_i)$ we define $\Psi(\langle b_1,\ldots, b_N,a_1,\ldots,a_r \rangle,s)$ with $a_i \not= 0$ to be
$\langle sb_1,\ldots, sb_N,a_1,\ldots, a_r\rangle$.  This clearly patches on the open covering $\{ V_i \}$ of $\bP Y_E$
 to determine $\Psi$ with 
$\Psi_{\bU_E \times \{ 1 \} } = id_{\bU_E}$ and $\Psi_{\bU_E \times \{ 0 \}} = s_E \circ p_E$.
This implies the second assertion.

Fix a splitting $s_0: J_E/J_E^2 \to J_E$ and set $\bA_H \ = \ S^\bu(Hom_k(J_E/J_E^2,J_E^2)^*)$.
We may parametrize the  splittings of $J_E \twoheadrightarrow J_E/J_E^2$ by 
$\bA_H$, a map $h \in Hom_k(J_E/J_E^2,J_E^2)$ corresponding to the splitting $s_0 +h$.
Define $\Delta: \bP Y_E \times \bA_H \to \bU_E$ to be the map sending the point 
$(\langle a_1,\ldots,a_r \rangle,h) \in \bP Y_E \times \bA_H$
to the point $\langle (s_0+h)_1(\ul a),\dots,(s_0+h)_N(\ul a),a_1,\ldots,a_r\rangle \in \bU_E$.  
Let $\ell: \bA^1 \to \bA_H$ be the line with $\ell(0) = s_E-s_0$ and 
$\ell(1) = s_{E^\prime} - s_0$.  Then $S_E \ = \ \Delta_{| \bP Y_E \times \ell(\bA_1)}$ provides the
map asserted in the third statement.
\end{proof}

We recall that $K_0(X)$ of a scheme $X$ is the Grothendieck group of locally free,
coherent $\cO_X$-modules; more explicitly,  $K_0(X)$ is the free abelian group on the set 
of isomorphism classes of locally free, coherent $\cO_X$-modules modulo the relations 
$[E] \ = \ [E^\prime] + [E^{\prime\prime}]$ associated to short exact sequences $0 \to E^\prime \to E \to E^{\prime\prime} \to 0$.  
It is useful to observe that $K_0(X)$ admits a natural ring structure given by tensor product of $\cO_X$-modules.  
If $X = \Spec A$ for a commutative ring $A$, then $K_0(X)$ is the group completion of the abelian monoid
of finitely generated, projective $A$-modules.

Passing from isomorphism classes of modules to their classes in $K_0$ loses
considerable information.  One example of this is that the computation of $K_0(\bP^n)$ is well known (isomorphic
as a ring to $\bZ[t]/t^{n+1}$, generated as an abelian group by the line bundles $\cO_{\bP^n}, \cO_{\bP^n}(1), 
\cdots,  \cO_{\bP^n}(n)$.   This stands in contrast to the long-standing geometric challenge of establishing the existence 
(or proving the non-existance)
of indecomposable vector bundles on $\bP^n$ of rank at least 2 and less than $n-1$; there are very few sporadic 
results addressing this problem.  

We also recall that $K_0^\prime(X)$ of a scheme $X$ is the Grothendieck group of 
coherent $\cO_X$-modules; more explicitly,  $K_0(X)$ is the free abelian group on the set 
of isomorphism classes of locally free, coherent $\cO_X$-modules modulo the relations 
$[\cF] \ = \ [\cF^\prime] + [\cF^{\prime\prime}]$ associated to short exact sequences 
$0 \to \cF^\prime \to \cF \to \cF^{\prime\prime} \to 0$.  

\begin{prop}
\label{prop:iso-KE}
The maps 
$$p_E^*: K_*(\bP Y_E) \ \to \ K_*(\bU_E), \quad p_E^*: K_*^\prime(\bP Y_E) \ \to \ K_*^\prime(\bU_E)$$ 
are isomorphisms.

In particular, the class in $K_0(\bP^{r-1})$ of $\ker\{ (\bP \Theta_{\bG_{a(1)}^{\times r},M})^j \}$ for a $kE$-module $M$ of 
constant $j$-rank (as considered in \cite{FP4})
does not depend upon the choice of generators for $E$.
\end{prop}

\begin{proof}
The isomorphisms $p_E^*$ are consequences of Quillen's homotopy invariance theorem for $K^\prime$ as
given in \cite[Prop 4.1]{Q} together with Proposition \ref{prop:aff-fib}(1).

Because the construction of $\ker\{ (\bP \Theta_{E,M} \}$ does not depend upon a choice of generators
for $E$, the class of $\ker\{ (\bP \Theta_{E,M} \}$ in $K_0(\bU_E)$ for $M$ a $kE$-modules of 
constant $j$-rank, Theorem \ref{thm:projbundle} implies that the class of  $\ker\{ (\bP \Theta_{\bG_{a(1)}^{\times r},M})^j \}$
in $K_0(\bP Y_E) \simeq K_0(\bU_E)$ also does depend upon a choice of generators for $E$.
\end{proof}

\vskip .2in


\section{Classes in K-theory}
\label{sec:Ktheory}

In this section, we briefly investigate the maps in $K$-theory and $K^\prime$-theory given by 
taking kernels, cokernels, and images of the maps $\Theta_{\tau,M}^j$ and $\bP \Theta_{\tau,M}^j$.
Here, $\tau$ is an (arbitrary) finite group, $M$ is a finitely generated $k\tau$-module, and $j$ is a positive
integer with $1 \leq j < p$.  As an introduction to equivariant $K$-theory, we recommend the survey article \cite{Mer} 
by A. Merkurjev.

We begin by introducing the following notation.  We restriction our attention to $K_0, K_0^\prime$, even though
the formalism of higher algebraic $K$-theory applies.

\begin{note}
\label{note:K}
We recall the following Grothendieck groups.  To avoid technicalities, we shall consider only Noetherian schemes.
\begin{enumerate}
\item
$R_k(\tau)$ (the Green ring of $k\tau$) is the group completion of the abelian monoid of isomorphism classes of
finite dimensional (left) $k\tau$-modules.
\item
$K_0(k\tau)$ is the group completion of the abelian monoid of isomorphism classes of finitely generated, projective
 (left) $k\tau$-module.
\item 
For a scheme $X$ equipped with an action of $\tau$, $K_0(\tau;X)$ is the Grothendieck group of $\tau$-equivariant,
locally free,  coherent $\cO_X$-modules  defined as the free abelian group on the set of isomorphism classes of such 
$\tau$-equivariant, locally free,  coherent $\cO_X$-modules modulo the relations
$[\cM] \ = \ [\cM_1] + [\cM_2]$ associated to short exact sequences of such modules of the form
$0 \to \cM_1 \to \cM \to \cM_2 \to 0$.  The ring $K_0(\tau;X)$ (with multiplication given by tensor product) is contravariant 
in $X$ and admits an evident restriction map for subgroups $\tau^\prime \hookrightarrow \tau.$
\item
For a scheme $X$ equipped with an action of $\tau$, $K_0^\prime(\tau;X)$ is the Grothendieck group of $\tau$-equivariant,
coherent $\cO_X$-modules  defined as the free abelian group on the set of isomorphism classes of such 
$\tau$-equivariant,  coherent $\cO_X$-modules modulo the relations
$[\cF] \ = \ [cF_1] + [\cF_2]$ associated to short exact sequences of such modules of the form
$0 \to \cF_1 \to \cF \to \cF_2 \to 0$.   Tensor product gives $K_0^\prime(\tau;X)$ the structure of a module over $K_0(\tau;X)$.
\end{enumerate}
\end{note}

The following is a direct consequence of Theorem \ref{thm:projbundle}

\begin{prop}
\label{prop:Ktau}
Consider a finite group $\tau$ and a finite dimensional $k\tau$-module $M$.  Then 
associating to $M$ the coherent sheaf $ker\{ \bP \Theta_{\tau,M}^j \}$ on $\bP X_\tau$ as in 
Theorem \ref{thm:projbundle} determines a group homorphism 
$$\kappa_j: R_k(\tau) \ \to \ K_0^\prime(\tau;\bP X_\tau), \quad 1 \leq j < p$$
which is contravariantly functorial with respect to injective maps $\sigma \to \tau$.
A similar statement applies if $ker\{ \bP \Theta_{\tau,M}^j \}$ is replaced by $coker\{ \bP \Theta_{\tau,M}^j \}$
or $im\{ \bP \Theta_{\tau,M}^j \}$.

Associating to a finite dimensional, projective $k\tau$-module $P$ the vector bundle defined as the restriction
to $\bU_\tau \subset \bP X_\tau$ of $ker\{ \bP \Theta_{\tau,P}^j \}$
similarly determines group homomorphisms (natural with respect to inclusions of finite groups $\sigma \to \tau$)
$$\kappa_j: K_0(k\tau) \ \to \ K_0(\tau:\bU_\tau), \quad 1 \leq j < p.$$
Once again, a similar statement applies if $ker\{ \bP \Theta_{\tau,P}^j \}$ is replaced by 
$coker\{ \bP \Theta_{\tau,P}^j \}$ or $im\{ \bP \Theta_{\tau,P}^j \}$.
\end{prop}

\begin{proof}
We restrict our attention to $ker\{ \bP \Theta_{\tau,M}^j \}$; the arguments for $coker\{ \bP \Theta_{\tau,M}^j \}$ 
and $im\{ \bP \Theta_{\tau,M}^j \}$ are essentially identical.

To establish $\kappa_j: R_k(\tau) \ \to \ K_0^\prime(\tau;\bP X_\tau), \quad 1 \leq j < p$, it suffices
to observe that the natural (with respect to $M$) construction of $ker\{ \bP \Theta_{\tau,M}^j \}$ commutes with 
direct sums.

The same argument applies to establish $\kappa_j: K_0(k\tau) \ \to \ K_0(\tau;\bU_\tau), \quad 1 \leq j < p$,
since exact sequences of projective $k\tau$-modules are always split.
\end{proof}

Our major tool in investigating $K_0$ is the following straight-forward consequence of J. Milnor's 
patching construction  \cite[\S 2]{Mil}. 

\begin{prop}
\label{prop:Milnor-patching}
The push-out squares of Remark \ref{remark:iterated}  which determine the colimit description $\bP X_\tau
 \simeq \varinjlim_{\cE(\tau)} \bP X_E$ assure that the data of a vector bundle $\cE_\tau$ on $\bP X_\tau$ is equivalent to
 the data of a vector bundle $\cE_E$ on each $\bP X_E$ together with compatible isomorphisms for each pair of 
 $E^{\prime\prime} \subset E, \ E^{\prime\prime} \subset E^\prime$ in $\cE(\tau)$ of the restrictions of 
 $\cE_E, \ \cE_{E^\prime}$ to $\bP X_{E^{\prime\prime}}$.

 Strictly analogous statements are valid for the colimit description of $\bP Y_\tau
 \simeq \varinjlim_{\cE(\tau)} \bP Y_E$ and $\bU_\tau
 \simeq \varinjlim_{\cE(\tau)} \bU_E$.
 \end{prop}
 
  \begin{proof}
 The assertion of patching follows from Milnor's patching results.  Namely, these results are extended to push-out squares 
 of the form (\ref{diag:push-outs}) in Remark \ref{remark:iterated} which involve projective varieties by observing that 
 one can patch locally by restricting to affine open coverings.  One then argues inductively to extend the patching
 from push-out squares to colimits indexed by $\cE(\tau)$ by realizing these colimits as iterated pushouts of the
 form (\ref{diag:push-outs}).
\end{proof}

 Using Proposition \ref{prop:Milnor-patching} and Milnor's Mayer-Vietoris exact sequence \cite[3.3]{Mil} we obtain
 the following theorem.
 
 \begin{thm}
 \label{thm:isoms}
Restriction maps determine $\tau$-equivariant isomorphisms
$$K_0(\bP X_\tau) \ \stackrel{\sim}{\to}  \ \varprojlim_{E \in \cE(\tau)} K_0(\bP X_E), \quad 
K_0(\bP Y_\tau) \stackrel{\sim}{\to} \varprojlim_{E \in \cE(\tau)} K_0(\bP Y_E), \quad 
K_0(\bU_\tau) \stackrel{\sim}{\to} \varprojlim_{E \in \cE(\tau)} K_0(\bU_E)$$
where the limits are taken in the category of abelian groups.

Moreover, push-forward map  maps determine $\tau$-equivariant isomorphisms
$$ \varinjlim_{E \in \cE(\tau)} K_0^\prime(\bP X_E) \ \stackrel{\sim}{\to}  \ K_0^\prime(\bP X_\tau)  \quad
\varinjlim_{E \in \cE(\tau)} K_0^\prime(\bP Y_E) \ \stackrel{\sim}{\to}  \ K_0^\prime(\bP Y_\tau), \\
\varinjlim_{E \in \cE(\tau)} K_0^\prime(\bU_E) \ \stackrel{\sim}{\to}  \ K_0^\prime(\bU_\tau). $$
Consequently, $p_\tau: \bU_\tau \ \to \ \bP Y_\tau$ induces a $\tau$-equivariant monomorphism
\begin{equation}
\label{eqn:ptau-iso2}
p_{\tau*}: K_0(\bU_\tau) \ \stackrel{\sim}{\to} \ K_0(\bP Y_\tau).
\end{equation}

\end{thm}

\begin{proof}
The $\tau$-equivariance of the asserted isomorphisms arises from naturality of the restriction maps and of the 
maps  $K_0(-) \to K^\prime(-)$.
 
 Milnor's Mayer-Vietoris exact sequence applied to push-outs
 \begin{equation}
 \label{diag:push-outs2}
\begin{xy}*!C\xybox{%
\xymatrix{
\bP X_{E^\prime} \ar[d] \ar[r]  & W  \ar[d] \\
\bP X_E  \ar[r] & Z ,
 }
}\end{xy}
\end{equation}
as in Remark \ref{remark:iterated} takes the following form
 $$K_1(\bP_{X_E}) \oplus K_1(W) \ \to \ K_1(\bP X_{E^\prime}) \to  K_0(Z) \ \to \ K_0(\bP X_E) \oplus K_0(W) \ \to \ K_0(\bP X_{E^\prime}) \to 0.$$
Since $K_1(\bP_{X_E}) \ \to \ K_1(\bP X_{E^\prime})$ is split surjective, these exact sequences 
yield split short exact sequences 
 \begin{equation}
 \label{eqn:shortZ}
0 \to  K_0(Z) \ \to \ K_0(\bP_{X_E}) \oplus K_0(W) \ \to \ K_0(\bP X_{E^\prime}) \to 0
\end{equation}
which we view as pull-back squares of abelian groups
\begin{equation}
\label{diag:pull-backs}
\begin{xy}*!C\xybox{%
\xymatrix{
K_0(Z) \ar[d] \ar[r] & K_0(\bP X_E) \ar[d] \\
K_0(W) \ar[r] & K_0(\bP X_{E^\prime}) .
 }
}\end{xy}
\end{equation}
In other words, the push-out squares of schemes of the form (\ref{diag:push-outs}) yield pull-back squares
of abelian groups.  This implies the identification of $K_0(X_\tau)$ as the abelian group obtained by
iterated pull-back squares of the form (\ref{diag:pull-backs}) which we identify as $\varprojlim_{E \in \cE(\tau)} K_0(\bP X_E)$.

The arguments for $\bP Y_\tau$ or $\bU_\tau$ in place of $\bP X_\tau$ are the same, except for notational changes.
We obtain the asserted isomorphism $p_\tau^*: K_0(\bP Y_\tau) \ \stackrel{\sim}{\to} \ K_0(\bU_\tau)$ as the 
colimit of the isomorphisms $p_E^*: K_0(\bP Y_E) \ \stackrel{\sim}{\to} \ K_0(\bU_E)$ of Proposition \ref{prop:iso-KE}.

To establish the $\tau$-equivariant  isomorphism 
$ \varinjlim_{E \in \cE(\tau)} K_0^\prime(\bP X_E) \ \stackrel{\sim}{\to}  \ K_0^\prime(\bP X_\tau)$
we consider once again push-out squares of the form (\ref{diag:push-outs2}).  We consider the map of localization
exact sequences
\begin{equation}
\label{diag:localization}
\begin{xy}*!C\xybox{%
\xymatrix{
K_1^\prime(\bP X_E \backslash \bP X_{E^\prime}) \ar[r] \ar[d]^{\simeq} & K_0^\prime(\bP X_{E^\prime}) \ar[d] \ar[r] & 
K_0^\prime (\bP X_E) \ar[r] \ar[d] & K_0^\prime(\bP X_E \backslash \bP X_{E^\prime}) \ar[r] \ar[d]^{\simeq}  & 0 \\
 K_1^\prime(Z\backslash W) \ar[r] & K_0^\prime(W) \ar[r] & K_0^\prime(Z) \ar[r] &
K_0^\prime(Z\backslash W) \ar[r] & 0
 }
}\end{xy}
\end{equation}
with indicated isomorphisms arising from the isomorphism of schemes 
$\bP X_E \backslash \bP X_{E^\prime} \ \stackrel{\sim}{\to} \ Z\backslash W$.  
By homotopy invariance of $K_*^\prime(-)$, we conclude that 
$K_1^\prime(\bP X_E \backslash \bP X_{E^\prime}) \to K_0^\prime(\bP X_{E^\prime})$ is the zero map.  Thus,
(\ref{diag:localization}) determines a map of short exact sequences.   A simple diagram chase now implies
that the middle square of (\ref{diag:localization}) is a push-out square of abelian groups.  Moreover, we verifiy
inductively (on the building of $\bP X_\tau$ as iterated push-out squares) that $K_0^\prime (\bP X_E) \to
K_0^\prime(Z)$ (as well as $K_0^\prime (W) \to K_0^\prime(Z)$) is injective.

Finally, observe that for $E^\prime \subset E$ that $J_{E^\prime}^2 = J_E^\prime \cap J_E^2$ so that 
$\bU_{E^\prime} \to \bU_E$ is a closed immersion (and therefore induces a push-forward map on $K_0^\prime(-)$).
Consequently, we can construct $\bU_\tau$ as an iterated push-out of proper maps
 \begin{equation}
 \label{diag:push-outs3}
\begin{xy}*!C\xybox{%
\xymatrix{
\bU_{E^\prime} \ar[d] \ar[r]  & W_U  \ar[d] \\
\bU_E  \ar[r] & Z_U.
 }
}\end{xy}
\end{equation}
We may thus repeat the argument for the $\tau$-equivariant isomorphism 
$\varinjlim_{E \in \cE(\tau)} K_0^\prime(\bP X_E) \ \stackrel{\sim}{\to}  \ K_0^\prime(\bP X_\tau)$ to 
conclude the $\tau$-equivariant isomorphism 
$ \varinjlim_{E \in \cE(\tau)} K_0^\prime(\bU_E) \ \stackrel{\sim}{\to}  \ K_0^\prime(\bU_\tau)$
\end{proof}

We employ the isomorphisms of Proposition \ref{prop:iso-KE}
for elementary abelian $p$-groups $E$ together with Theorem \ref{thm:isoms} to 
conclude the following corollary.

\begin{cor}
\label{cor:pullback-iso}
For any finite group $\tau$, $p_\tau: \bU_\tau \ \to \ \bP Y_\tau$ induces  $\tau$-equivariant isomorphisms
\begin{equation}
\label{eqn:ptau-iso}
p_\tau^*: K_0(\bP Y_\tau) \ \stackrel{\sim}{\to} \ K_0(\bU_\tau), \quad
p_{\tau}*: K_0^\prime(\bP Y_\tau) \ \stackrel{\sim}{\to} \ K_0(\bU_\tau).
\end{equation}
\end{cor}

\begin{proof}
We obtain the asserted isomorphisms using the squares
 \begin{equation}
 \label{diag:Ksquare}
\begin{xy}*!C\xybox{%
\xymatrix{
K_0(\bP Y_\tau) \ar[d]^\simeq \ar[r]^{p_\tau^*} & K_0(\bU_\tau) \ar[d]^\simeq \\
\varprojlim_{E \in \cE(\tau)} K_0(\bP Y_E) \ar[r]^\simeq & \varprojlim_{E \in \cE(\tau)} K_0(\bU_E),
 }
}\end{xy}
\end{equation}
 \begin{equation}
 \label{diag:Ksquare2}
\begin{xy}*!C\xybox{%
\xymatrix{
\varinjlim_{E \in \cE(\tau)} K_0^\prime(\bP Y_E) \ar[d]^\simeq \ar[r]^\simeq & \varinjlim_{E \in \cE(\tau)} 
K_0^\prime(\bU_E) \ar[d]^\simeq \\
K_0^\prime(\bP Y_\tau) \ar[r]^{p_\tau^*} & K_0^\prime(\bU_\tau),
 }
}\end{xy}
\end{equation}
whose vertical isomorphisms are given by Theorem \ref{thm:isoms} and whose isomorphisms involving
limits are given by the isomorphisms of Proposition \ref{prop:iso-KE} and their naturality.
\end{proof}

The isomorphisms $K_0(\bP Y_\tau) \stackrel{\sim}{\to} \varprojlim_{E \in \cE(\tau)} K_0(\bP Y_E)$ 
and  $\varinjlim_{E \in \cE(\tau)} K_0^\prime(\bP Y_E) \ \stackrel{\sim}{\to}  \ K_0^\prime(\bP Y_\tau$
of Theorem \ref{thm:isoms} easily imply the following corollary once one recalls the 
construction of limits and colimits of diagrams of abelian groups.

\begin{cor}
\label{cor:inj-restr}
Let $\tau$ be a finite group and let $E_1,\ldots E_s$ be a list of the maximal
elementary abelian $p$-subgroups of $\tau$.  Then the natural maps
$$K_0(\bP Y_\tau) \ \to \ \bigoplus_{i=1}^s K_0(\bP Y_{E_i}), \quad K_0(\bU_\tau) \ \to \ 
\bigoplus_{i=1}^s K_0(\bU_{E_i})$$
are injective, and the natural maps
$$ \bigoplus_{i=1}^s K_0^\prime(\bP Y_{E_i}) \to K_0^\prime(\bP Y_\tau), \quad 
\bigoplus_{i=1}^s K_0^\prime(\bU_{E_i}) \to K_0^\prime(\bU_\tau)  $$
are surjective.
\end{cor}

\vskip .2in


\section{$X_{\bG(\bF_p)} \ \to \cN_p(\fg)$}
\label{sec:rational}

In this section, we consider the finite group $\tau = \bG(\bF_p)$ of $\bF_p$-rational points of a (connected) reductive
algebraic group $\bG$ defined over $\bF_p$ with restricted Lie algebra $\fg$ and Frobenius kernels $\bG_{(r)}$.  
For a given reductive group $\bG$, we require that the prime $p$ be ``suitable" for $\bG$ in order to utilize the
exponential map $exp: \cN_p(\fg) \to \cU_p(\bG)$ constructed by Sobaje, a ``Springer isomorphism" with good 
properties.   As a first step, we reframe the work of Carlson, Lin, and Nakano in \cite{CLN} using our constructions
of previous sections. Proposition \ref{prop:cln} relates $\tau$ to $\bG_{(1)}$ (or equivalently, to the restricted Lie algebra 
$\fu(\fg)$) by exhibiting a $\tau$-equivariant isomorphism $\cL_\tau: Y_\tau \stackrel{\sim}{\to}  Y_\fg$.
In order to deepen this relationship by considering $\bG_{(r)}$ for $r > 1$, we next construct in
Proposition \ref{prop:ul-ell} the natural map $\ell_E^{(r)}: X_E \to V_r(\ul E)$, providing a relationship between $X_E$ for an elementary abelian $p$-group $E$ and the
variety $V_r(\ul E)$ of 1-parameter subgroups of the associated vector group $\ul E$.  The naturality of our 
constructions allows us to construct $\iota_\tau^{(r)}: X_\tau \to V_r(\bG)$ in Proposition \ref{prop:cln-extend} 
which refines the relationship established for $r=1$.
Theorem \ref{thm:Lie-compare} utilizes $\iota_\tau^{(r)}$ to compare the vector bundles constructed for $\tau$-modules
in Theorem \ref{thm:projbundle} to those construction by Pevtsova and the author on $\Proj k[V_r(\bG)]$ for certain
rational $\bG$ modules.

 We begin by constructing analogues for Lie algebras of the constructions given in Section \ref{sec:tau} for finite groups.   
 We remark in passing that the Weil restriction  $\mathfrak R_{\bF_q}(\bG \otimes \bF_q)$ of $\bG$ is a reductive
 group over $\bF_p$ whose group of $\bF_q$ points equals $\bG(\bF_q)$, $q = p^d$.  Thus, as in \cite{F1} (see also \cite{F2}),
 much of the following discussion applies to finite groups of the form $\bG(\bF_q)$.

The Lie algebra $\fg$ of $\bG$ is the base change
$\fg  \ = \ \fgp \otimes k$ of the Lie algebra of $G_{\bF_p}$.   We shall utilize 
the (finite) partially ordered set $\cE(\fgp)$ of elementary Lie subalgebras $\epsilon_{\bF_p} \subset \fgp$
ordered by inclusion.  We recall that an elementary Lie algebra over $k$ is a finite dimensional vector space 
$\epsilon$ over $k$ with
0 Lie bracket and 0 $p$-th power operation.  The restricted enveloping algebra $\fu(\epsilon)$ 
of such a Lie algebra is the truncated polynomial algebra $S^\bu(\epsilon)/(X^p, X \in \epsilon)$, 
isomorphic to the group algebra $kE$ of an elementary abelian $p$-group of rank equal
to the dimension of $\epsilon$.

\begin{prop}
\label{prop:Yg}
Define 
$$B_\fg \equiv \ \varprojlim_{\epsilon \in \cE(\fgp)} S^\bu((\epsilon \otimes_{\bF_p} k)^*).$$  
This is a finitely generated commutative $k$-algebra (depending contravariantly on $\fg$)
equipped with a natural grading and a (conjugation) action by $\tau$.
We further define
$$Y_\fg \equiv  \Spec B_\fg \ = \ \varinjlim_{\epsilon \in \cE(\fgp)} Y_\epsilon, \quad \quad Y_\epsilon \equiv 
\Spec S^\bu((\epsilon \otimes_{\bF_p} k)^*) \equiv \bA_\epsilon.$$
So defined, $Y_\fg$ admits a natural $\tau$-equivariant embedding into $\cN_p(\fg)$, the closed subvariety of $\bA_\fg$
consisting of $p$-nilpotent elements.  (On $\cN_p(G)$, $\tau$ acts via the restriction of the adjoint action of $G$ on $\cN_p(\fg)$.)
\end{prop}

\begin{proof}
The grading of $B_\fg$ and the statement that $Y_\fg$ is the indicated colimit in the category of
schemes is proved exactly as in the proof of Proposition \ref{prop:limits}.  The actions of $\tau$ on
$B_\fg$ and $Y_\fg$ are induced by the action of $\tau$ on $\cE(\fgp)$.

By definition, $\epsilon   \otimes_{\bF_p} k$ is embedded in $\fg$ as a restricted subalgebra for
each $\epsilon \in \cE(\fgp)$ and thus is embedded in $\cN_p(\fg)$.  Since both
$Y_\fg$ and $\cN_p(\fg)$ are stable under the conjugation action of $\tau$, this embedding is $\tau$-invariant.
\end{proof}

\begin{remark}
The dimension of $Y_\fg$ equals the maximum of the dimensions of the elementary subalgebras $\epsilon_{\bF_p} 
\in \cE(\fgp)$ which is typically much smaller than $dim (\cN_p(\fg))$.  For example, if $\bG = GL_{2n}$, then the maximal
elementary abelian $p$-subgroups of $GL_{2n}(\bF_p)$ have rank $n^2$ where the dimension of $\cN_p(\gl_{2n})$ is
$4n^2 - 2n$.
\end{remark}

In what follows, we shall require $\bG$ to satisfy the following hypotheses.

\begin{hypothesis}
\label{hypoth:G}
Assume that $\bG$ is a reductive algebraic group over $k$ defined over $\bF_p$ such that $p$
is ``suitable" for $\bG$ in the sense of \cite[Defn 1.9]{F2}.  Namely, $p$ is suitable for $\bG$ if 
\begin{itemize}
\item 
$p$ does not divide $\pi_1(\bbH)$ for any factor $\bbH$ of $[\bG,\bG]$ of type $A_\ell$;
\item
$p \not= 2$ if some factor $\bbH$ of $[\bG,\bG]$ is of type $B_\ell, C_\ell$ or $D_\ell$, $p \not= 2$;
\item 
$p \not= 2,3$ if some factor $\bbH$ of $[\bG,\bG]$ is of type $G_2, F_4, E_6$ or $E_7$;
\item
$p \not= 2, 3, 5$ if some factor $\bbH$ of $[\bG,\bG]$ is of type $E_8$, $p \not= 2, 3, 5$.
\end{itemize}
\end{hypothesis}

Under the conditions of Hypothesis \ref{hypoth:G}, the following theorem of P. Sobaje gives
a good understanding of ``Springer isomorphisms" as we now recall.  In the statement below, $\cU_p(\bG)$
is the closed subvariety of $\bG$ whose $k$-points are the elements $x \in \bG(k)$  with $x^p = e$.

\begin{thm} (P. Sobaje \cite[Thm 3.4]{S3}; see also \cite{S2})
\label{thm:sobaje}
For $\bG$ as in Hypothesis \ref{hypoth:G}, there is a unique $\bG$-equivariant bjjection
$$exp: \cN_p(\fg) \ \stackrel{\sim}{\to} \ \cU_p(\bG)$$
which is suitably behaved when restricted to unipotent radicals of parabolic subgroups of $G$
of nilpotence class $< p$.  Moreover, $exp$ satisfies
\begin{enumerate}
\item
$[X,Y] = 0 \ \iff \ (exp(X),exp(Y)) = 1.$
\item 
$exp$ is defined over $\bF_p$.
\item
If $[X,Y] = 0$, then $exp(X+Y) = exp(X)\cdot exp(Y)$.
\end{enumerate}
\end{thm}

Using the Springer isomorphism of Theorem \ref{thm:sobaje}, J. Warner established
the following theorem.  In fact, Warner's theorem looked at the enriched categories (considered by
Quillen in \cite{Q2})  with the same 
objects but more maps (namely, inclusions followed by adjoint actions).

\begin{prop} (J. Warner \cite[3.2.5]{JW})
\label{prop:warner}
The Springer isomorphism of  Theorem \ref{thm:sobaje} induces an equivalence of 
categories with $\tau$-action whose inverse we denote by
$$\ell:  \cE(\tau) \ \stackrel{\sim}{\to} \ \cE_{\bF_p}(\fg).$$
\end{prop}

The following proposition is a recasting of \cite{CLN} (see also \cite{FP2}).

\begin{prop}
\label{prop:cln}
For $E \in \cE(\tau)$, we define 
$$\ell_E: J_E \ \to  \epsilon_E \otimes k, \quad \sum_{g \in E} a_g (g-e) \ \mapsto \ \sum_{g \in E} a_g\ell(g)$$
where $\epsilon_E \equiv \ell(E)$.
Then, $\ell_E$ sends $J_E^2$ to 0, therefore determining the isomorphism 
$\cL_E: J_E/J_E^2 \ \stackrel{\sim}{\to} \ \epsilon_E \otimes k$.

The colimit indexed by $\cE(\tau)$ of the associated maps of $k$-schemes  
$$\cL_E: Y_E \ \equiv \  \bA_{J_E/J_E^2} \ \to \  \bA_\epsilon \equiv Y_\epsilon$$
is a $\tau$-equivariant isomorphism
$$\cL_\tau: Y_\tau \ \stackrel{\sim}{\to} \ Y_\fg.$$

So defined, $\cL_\tau$ determines a $\tau$-equivariant embedding
\begin{equation}
\label{eqn:embed}
\bP \iota_\tau: \bP Y_\tau \ \hookrightarrow \ \Proj (k[\cN_p(\fg)]).
\end{equation}
\end{prop}

\begin{proof}
The map $\ell_E$ was constructed in \cite[5.8]{FP1}.  The fact that $\ell_E(J_E^2) = 0$ follows
from the properties of $exp$ and thus of $\ell_E$ given in Theorem \ref{thm:sobaje} and 
the identity 
\begin{equation}
\label{eqn:identity}
(gh -e) \ = \ (g-e) + (h-e) + (g-e)(h-e) \ \in \ J_E.
\end{equation}
Namely, we conclude that the element $(g-e)(h-e) \in J_E^2$ is sent via $\ell_E$ to 
$\ell_E(gh) - \ell_E(g) - \ell_E(h) =0$.

The isomorphism of colimits follows from the isomorphism of indexing categories given in
Proposition \ref{prop:warner}.  The fact that this isomorphism is $\tau$ invariant follows from
the observation that the $\tau$ actions arise from the actions of $\tau$ on $\cE_{\bF_p}(\fg)$
and $\cE(\tau)$ and that these are $\tau$ equivalent by  Proposition \ref{prop:warner}. 

Finally, $\bP \iota_\tau$ is given as the projectivization of the composition of $\cL_\tau$ and the embedding
$Y_\fg \hookrightarrow \cN_p(\fg)$ of Proposition \ref{prop:Yg}.
\end{proof}

As shown in \cite{SFB1}, the variety
$V_r(\bG_{a(r)})$ of 1-parameter subgroups of height $r$ of $\bG_a$ is naturally isomorphic to $\Spec k[X^{(0)},\ldots,X^{(r-1)}]$:
each 1-parameter subgroup $\bG_{a(r)} \to \bG_a$ given by a map $k[T] \to k[T]/T^{p^r}$ sending
$T$ to a polynomial of the form $p(T) = c_0T + c_1T^p + \cdots + c_{r-1}T^{p^{r-1}}$; the function $X^{(i)}
\in k[X^{(0)},\ldots,X^{(r-1)}]$ 
is given the grading determined by the action of $\bA^1$ on $V_r(\bG_{a(r)})$  sending
$s \in \bA^1(k), p(T) \in V_r(\bG_{a(r)})$  to the 1-parameter subgroup
corresponding to the polynomial $p(sT)$.  With this grading, $X^i$ has ``weight" $p^i$.

In the next proposition, we extend Proposition \ref{prop:cln} to provide a graded map $\ell_E^{(r)}: \bA_{J_E} \to V_r(\ul E)$ 
lifting $\ell_E^{(1)} \equiv \ell_E: J_E \to \epsilon_E\otimes k$.

\begin{prop}
\label{prop:ul-ell}
Fix some $r > 0$.
Let $E$ be an elementary abelian $p$-group of rank $s$, set $\ul E_{\bF_p} \simeq \bG_{a,\bF_p}^{\times s}$,
set $\epsilon_{E_{\bF_p}} = Lie(\ul E_{\bF_p})$, $\ul E = \ul E_{\bF_p} \otimes k$.  We 
identify $E$ with the $\bF_p$-points of $\ul E_{\bF_p}$, and choose a basis $g_1,\ldots, g_s$ for $E$,
setting $x_j = g_j -e \in J_{E_{\bF_p}}/J_{E_{\bF_p}}^2 
\simeq \epsilon_{E_{\bF_p}}$.  We identify $V_r(\ul E)$ with the polynomial algebra 
$k[X_j^{(i)}, 1 \leq j \leq s, 0 \leq i < r]$ with grading given by $wt(X_j^{(i)}) = p^i$.
We define
\begin{equation}
\label{eqn:ell-r}
\ell_E^{(r)}: X_E \equiv \bA_{J_E} \ \to \ V_r(\ul E), 
\quad (\ell_E^{(r)})^*(X_j^{(i)}) = ((g_j-e)^{\vee})^{p^i}.
\end{equation}
So defined,  $\ell_E^{(r)}$ satisfies the following properties:
\begin{enumerate}
\item
$\ell_E^{(r)}(a(g_i-e)) \ = \ (ax_i,a^px_i,\ldots,a^{p^{r-1}}x_i)$ for any $a \in k$.
\item 
$\ell_E^{(r)}(J_E^2) = 0$, so that $\ell_E^{(r)}$ factors through $Y_E$.
\item
The composition of $\ell_E^{(r)}$ with the restriction map $V_r(\ul E) \to V_{r-1}(\ul E)$ equals $\ell_E^{(r-1)}$.
\item
If $E^\prime \subset E$, then $\ell_E^{(r)}$ restricts to $\ell_{E^\prime}^{(r)}$.
\item
$(\ell_E^{(r)})^*$ is a  map of graded $k$-algebras, inducing 
$$\bP \ell_E^{(r)}: \bP Y_E = \Proj S^\bu((J_E/J_E^2)^*)  \ \to \ \Proj k[V_r(\ul E)]$$
which is the composition of a $p$-isogeny and a closed immersion.
\end{enumerate}
\end{prop}

\begin{proof}
Assertion (1) follows from the observation that $((g_j-e)^{\vee})^{p^i}(a(g-e)) = ((g_j-e)^{\vee}(a(g-e)))^{p^i} 
$ equals $a^{p^i}$ if $g=g_j$ and $0$
if $g \not= g_j$.  We verify assertion (2) by observing that  $((g_j-e)^{\vee})^{p^i}((g-e)(h-e) =0$ for all $g, h \in E$.
The restriction map $V_r(\ul E) \to V_{r-1}(\ul E)$ is given on coordinate functions by the natural inclusion 
$k[X_j^{(i)}, 1 \leq j \leq s, 0 \leq i < r-1] \to k[X_j^{(i)}, 1 \leq j \leq s, 0 \leq i < r]$ which establishes 
assertion (3). 

Assertion (4) follows from assertion (3) which tells us that that $\ell_E^{(r)}$ factors through $\bA_{J_E} \ 
\twoheadrightarrow \bA_{J_E/J_E^2}$ and thus does not depend upon a splitting of $J_E \twoheadrightarrow J_E/J_E^2$.

The algebra map  $(\ell_E^{(r)})^*: k[V_r(\ul E)] \to S^\bu((J_E/J_E^2)), \ X_j^{(i)} \mapsto ((g_j-e)^\vee)^{p^i}$ factors as 
$$k[V_r(\ul E)] \ \twoheadrightarrow
 \ S^\bu(\{ ((g_j-e)^\vee)^{p^i}\}) \ \hookrightarrow S^\bu((J_E/J_E^2)),$$
whose associated map of schemes is a composition of a $p$-isogeny and a closed immersion.
\end{proof}

One special aspect of the algebraic group $\ul E$ is the existence of an inverse (given by truncation) 
of the natural inclusion $k \ul E_{(r)} \to k \ul E$ (defined as the ``dual" of the quotient map $k[\ul E] \to k[\ul E_{(r)}]$).
We denote the composition of the natural inclusion $kE \hookrightarrow$ with this inverse by 
$tr_E: kE \ \to \ k \ul E_{(r)}.$

The following proposition relates the universal $p$-nilpotent operator $\Theta_E \in A_E \otimes kE$ to 
the universal $p$-nilpotent operator $\Theta_{\ul E_{(r)}}$.

\begin{prop}
\label{prop:thetaE-restrict}
Retain the notation of Proposition \ref{prop:ul-ell}.   Modulo terms in $Rad^2(k \ul E_{(r)})$, the universal $p$-nilpotent 
operator $\Theta_{\ul E_{(r)}}$ equals 
\begin{equation}
\label{eqn:trunc-theta}
\ol{ \Theta_{\ul E_{(r)}}} \ \equiv \ \sum_{j=1}^s \sum_{i=0}^{r-1} (X_j^{(i)})^{p^{r-i-1}} \otimes u_{r-i-1}^j  \ \in 
\ k[V_r(\ul E)] \otimes k \ul E_{(r)}.
\end{equation}
Modulo terms in $Rad^2(kE)$, $\Theta_E$ equals
\begin{equation}
\label{eqn:trunc-theta-E}
\ol{\Theta_E} \ \equiv \ \sum_{j=1}^s (g_j-e)^\vee \otimes g_j-e \ \in A_E \otimes kE
\end{equation}
for any choice of basis $\{ g_1,\ldots,g_r\}$ of $E$ (so that $\ol{\Theta_E}$ is the restriction  of 
$\sum_{j=1}^s \ol {(g_j-e)}^\vee \otimes g_j-e \in B_E \otimes kE$ along $s_E^*: A_E \to B_E$).

We set \ $\ol{\Theta_E}^{(r-1)} \ \equiv \ \sum_{j=1}^s ((g_j-e)^\vee)^{p^{r-1}} \otimes g_j-e$.
Then 
\begin{equation}
\label{eqn:theta-equal}
((\ell_E^{(r)})^*\otimes 1)(\ol{ \Theta_{\ul E_{(r)}}}) \ = \ (1\otimes tr_{E*})(\ol{\Theta_E}^{(r-1)}) \ \in \ B_E \otimes  k \ul E_{(r)}.
\end{equation}
In other words, we have the following commutative square 
\begin{equation}
 \label{diag:theta-commute}
\begin{xy}*!C\xybox{%
\xymatrix{
k[\ul E_{(r)}] \ar[d]^{tr_E^*} \ar[r]^{\ol{ \Theta_{\ul E_{(r)}}}} & k[V_r(\ul E)] \ar[d]^{(\ell_E^{(r)})^*} \\
k[E] \ar[r]^{{\ol \Theta_E^{(r-1)}}} & B_E .
 }
}\end{xy}
\end{equation}
\end{prop}

\begin{proof}
The explicit description of $\ol{ \Theta_{\ul E_{(r)}}}$ is given towards the end of the proof of \cite[Prop 6.5]{SFB1}.
The description of $\ol {\Theta_E}$ is implicit in the first sentence of the proof of Proposition \ref{prop:radical-map}.

Since $B_E$ is reduced, to verify the commutativity of (\ref{diag:theta-commute}) it suffices 
to consider the composition of the two maps $k[\ul E_{(r)}] \to B_E$ with the valuation maps 
$c(g_j-e): B_E \to K$ for each $j, 1 \leq j \leq s$ and each $c \in K$ for field extensions $K/k$ (these are the
geometric points of $Y_E = \Spec B_E$).  An explicit computation verifies
that each of these compositions sends $X_j^{(i)}$ to  $c^{p^{r-1}}u_i^j \in K\ul E_{(r)}$.
 \end{proof}
 
 As for $\Theta_\tau$ in Definition \ref{defn:Theta-tauM} and $\bP \Theta_\tau$ in Definition \ref{defn:bPTheta},
 the operators of the previous proposition produce determine module homomorphisms.  These are the subject of
 the next proposition. 
 
 \begin{prop}
 \label{prop:theta-homotopy}
 Retain the notation of Proposition \ref{prop:thetaE-restrict}.   For any finite dimensional $kE$-module, the
 kernel, cokernel, and image of $\bP \ol{\Theta_{E,M}}^{(r-1)}$ are coherent sheaves on $\bP X_E$ which 
 are the $(r-1)$-st Frobenius twists of kernel, cokernel, and image of $\bP \ol{\Theta_{E,M}}$.
 
 For any $kE$-module $M$ of constant $j$-rank, the kernel, cokernel, and image of
 \begin{equation}
 \label{eqn:homotopyE}
 (\bP \ol{\Theta_{E,M}}^{(r-1)})^j+ t((\bP \Theta_{E,M}^{(r-1)})^j-(\bP \ol{\Theta_{E,M}}^{(r-1)})^j):
 \cO_{\bP X_E}[t] \otimes M \ \to \ \cO_{\bP X_E}(jp^{r-1})[t] \otimes M
 \end{equation}
 are coherent sheaves on $\bP X_E \times \bA^1$ which restrict to vector bundles on $\bU_E \times \bA^1$. 
 
 Similarly, If $M$ is a $\ul E_{(r)}$-module of constant $j$-rank, then the kernel, cokernel, and image of
 \begin{equation}
 \label{eqn:homotopyEr}
  (\bP \ol{\Theta_{\ul E_{(r)},M}})^j +
 t((\bP \Theta_{\ul E_{(r)},M})^j- (\bP\ol{ \Theta_{\ul E_{(r)},M}})^j  )
 \end{equation}
 are vector bundles on $\Proj k[V_r(\ul E)] \times \bA^1$.
 
 Consequently, assuming constant $j$-rank, the classes in $K_0(\bU_E)$ of  
 $$ker\{(\bP \ol{\Theta_{E,M}}^{(r-1)})^j\}, \quad  ker\{ (\bP \Theta_{E,M}^{(r-1)})^j\}$$ are equal,
 as are the classes in $K_0(\Proj k[V_r(\ul E)])$ of 
  $$ker\{(\bP\ol{ \Theta_{\ul E_{(r)},M}})^j \}, \quad ker\{(\bP \Theta_{\ul E_{(r)},M})^j\}.$$ 
 The analogous equalities with kernel replaced by image and with kernel replaced by cokernel 
  are also valid.
 \end{prop}
  
 \begin{proof}
 The identification of the kernel, image, cokernel of $\bP \ol{\Theta_{E,M}}^{(r-1)}$ as $r-1$-st Frobenius twists of
 the kernel, image, cokernel of $\bP \ol{\Theta_{E,M}}$ follows directly from the observation that $\ol{\Theta_E}^{(r-1)}$
 is the pull-back along the $r-1$-st Frobenius $F^{r-1}: X_E \to X_E$ of $\ol{\Theta_E}$.
 
 We readily verify that (\ref{eqn:homotopyE}) is a well defined map of $M \otimes \cO_{\bP X_E \times \bA^1}$-modules
 arising from a homogeneous $A_E[t]$-endomorphism of $M\otimes A_E[t]$ of degree $jp^{r-1}$.
 To verify that the asserted kernels, cokernels, and images are vector bundles, it suffices by Proposition \ref{prop:locfree}
 to verify that these have fibers at points of $\bP X_E \times \bA^1$ of dimension equal to the constant $j$-rank of $M$.
 This follows from the observation that the fibers of $\bP \Theta_{E,M}^{(r-1)})^j-(\bP \ol{\Theta_{E,M}}^{(r-1)})^j$
 at points of $\bP X_E$ are all of smaller dimension.  This latter fact is a consequence of the fact that at each $K$-point
 of $\bP X_E$ the action of this difference operator is given by an element of $Rad^{j+1}(KE)$ on $K\otimes M$ which
 necessarily has rank less than the (constant) rank of the $j$-th power of elements of $Rad(KE)\backslash Rad^2(KE)$ 
 on $K\otimes M$.
 
 Observe that the vector bundles of $\bU_E$ given by restricting 
 $ker\{(\bP \ol{\Theta_{E,M}}^{(r-1)})^j\}, \\  ker\{ (\bP \Theta_{E,M}^{(r-1)})^j\}$ are restrictions to $\bU_E \times \{ 0\}, \ 
 \bU_E \times \{ 1\}$ of the kernel vector bundle on $\bU_E \times \bA_1$ associated to (\ref{eqn:homotopyE}).  
 Thus, homotopy invariance of $K_0(-)$ tells 
 us that the classes of these vector bundles in $K_0(\bu U_E)$ are equal.  Replacing kernels by either images or cokernels 
 gives the corresponding equality of classes in $K_0(\bu U_E)$.
 
 The verification of the analogous statements for $\ul E_{(r)}$-modules $M$ of constant $j$-rank requires only 
 notational changes of the verification for $kE$-modules.
  \end{proof}

We recall  the following definition of  {\it exponential degree} for a rational $\bG$-module.
(See also the definition of ``nilpotent degree" of $M$ in \cite[Defn 2.5]{F2}.)

 \begin{defn} (\cite[Defn 4.5]{F3})
 Let $\bG$ be a linear algebraic group and let $M$ be a  rational $\bG$ module.  Then 
 $M$ is said to have exponential degree $< p^r$ if the coaction 
$$M \ \to \ M \otimes k[\bG] \ \stackrel{\psi^*}{\to} \ M \otimes k[\bG_a] =  M \otimes k[T]$$
for any 1-parameter subgroup $\psi: \bG_a \to \bG$
 lies in $M \otimes k[T]_{< p^r}$, where $k[T]_{< p^r} \subset k[T]$ is the sub-coalgebra of polynomials 
 of degree $< p^r$.
 \end{defn}
 
 \vskip .1in
 
 Every finite dimensional rational $\bG$-module $M$ has finite exponential degree provided
 that 1-parameter subgroups $\bG_a \to \bG$ are parametrized by finite sequences of elements of 
 $\fg$.  This is the case for $\bG$ equal to $\ul E$ or for $\bG$ satisfying Hypothesis \ref{hypoth:G} .

\begin{prop}
\label{prop:E-compare}
Retain the notation of Propositions \ref{prop:ul-ell} and \ref{prop:thetaE-restrict}; in particular, $E$ is an
elementary abelian $p$-group of rank $s$ and $\ul E \simeq \bG_a^{\times s}$.  Let $M$ be a finite
dimensional rational $\ul E$-module whose restriction to $\ul E_{(r)}$ has non-zero constant $j$-rank
and which has exponential degree $< p^r$.  Thus, the restriction of $M$ to $E$ has constant $j$-rank 
equal to constant $j$-rank of $M$ restricted to $\ul E_{(r))}$.

The classes in $K_0(\bU_E)$  of the $(r-1)$-st Frobenius twists of the vector bundles 
$$ker\{ (\bP \Theta_{E,M})^j \}, \quad coker\{ (\bP \Theta_{E,M})^j \}, \quad im\{ (\bP \Theta_{E,M})^j\} $$
on $\bU_E$ constructed in Theorem \ref{thm:projbundle} are equal to the images under
$$(\bP \ell_E^{(r)})^*: K_0(\Proj k[V_r(\ul E)]) \ \to \ K_0(\bU_E)$$ 
of the classes of the corresponding vector bundles 
$$ker\{\bP \Theta_{\ul E_{(r)},M}^j \}, \quad coker\{\bP \Theta_{\ul E_{(r)},M}^j \}, \quad im\{\bP \Theta_{\ul E_{(r)},M}^j \}$$
on $\Proj k[V_r(\ul E)] $  constructed in \cite{FP4}. 
 \end{prop}
 
 \begin{proof}
 By \cite[Thm 4.11]{F2}, if $M$ has constant $j$-rank for $\bE_{(r)}$ and has exponential degree $< p^r$, 
 then $M$ has constant $j$-rank for $kE$.  By Proposition \ref{prop:theta-homotopy}, it suffices to replace the
 $(r-1)$-st Frobenius twists of kernel, image, and cokernel of  $(\bP \Theta_{E,M})^j$ by the kernel, image, and
 cokernel of $(\bP \ol{\Theta_{E,M}}^{(r-1)})^j$, and it suffices to replace kernel, image, and cokernel of  
 $(\bP \Theta_{\ul E_{(r)},M})^j$ by kernel, image, and cokernel of  $(\bP \ol{\Theta_{\ul E_{(r)},M}})^j$.
 
 Under the hypothesis that $M$ has exponential degree $< p^r$, we claim that the equality (\ref{eqn:theta-equal})
 implies that the kernel, image, and cokernel of $\bP \ol{\Theta_{E,M}}^{(r-1)}$ are isomorphic as coherent 
 sheaves on $\bP X_E$ to the result of
 applying $(\bP \ell_E^{(r)})^*$ to the  kernel, image, and cokernel of $\bP {\ol \Theta_{\ul E_{(r)},M}}^j$.
Namely, this hypothesis implies that $\bP \ol{\Theta_{E,M}}^{(r-1)}$ equals 
$\bP (1\otimes tr_{E*})(\ol{\Theta_{E,M}}^{(r-1)})$ which by (\ref{eqn:theta-equal}) equals the restriction along
$\bP \ell_E^{(r)}: \bP X_E \to \Proj k[V_r(\ul E)]$ of $\bP \ol{\Theta_{\ul E_{(r)},M}})$.
  \end{proof}

We extend the construction of $\ell_E^{(r)}$ of Proposition \ref{prop:ul-ell} to arbitrary finite groups using 
the exponentiation map $exp$ of Theorem \ref{thm:sobaje}.

\begin{prop}
\label{prop:cln-extend}
Retain the hypotheses and notation of Proposition \ref{prop:cln}. Let $E \subset \bG(\bF_p)$ be an elementary abelian
$p$-group corresponding as in Proposition \ref{prop:warner} to the elementary subalgebra $\epsilon_E \subset \fg_{\bF_p}$
and define the embedding $i_{\ul E}:\ul E \ \hookrightarrow \bG$ as the embedding of the image of the map 
$exp: \epsilon_E \otimes k \to \bG$.   Then 
the composition 
$$(i_{\ul E})_*\circ \ell_E^{(r)}: X_E \ \to \ V_r(\ul E) \ \to \ V_r(\bG)$$
restricts to $(i_{\ul E^\prime})_*\circ \ell_{E^\prime}^{(r)}: X_{E^\prime} \ \to \ V_r(\bG)$ for any $E^\prime \subset E$ in $\cE(\tau)$.

Consequently, the colimit indexed by $E \in \cE(\tau)$ of the maps $(i_{\ul E})_* \circ \ell_E^{(r)}$,
$$\ell_\tau^{(r)} \equiv \varinjlim_{E \in \cE(\tau} ((i_{\ul E})_* \circ \ell_E^{(r)}):  X_\tau \equiv 
\varinjlim_{E \in \cE(\tau)} \bA_{J_E} \ \to \ V_r(\bG)$$
is a well defined, $\tau$-equivariant, graded map which projects to $\ell_\tau^{(r-1)}$ and 
factors through the projection $p_\tau: X_\tau \to Y_\tau$.  

Furthermore, $\ell_\tau^{(r)}$ determines the $\tau$-equivariant map
$$\bP \ell_\tau^{(r)}: \bP Y_\tau \ \to \ \Proj k[V_r(\bG)]$$
which is the composition of a $p$-isogeny and a closed immersion.
\end{prop} 

\begin{proof}
The fact that $(i_{\ul E})_*\circ \ell_E^{(r)}$ restricts to $(i_{\ul E^\prime})_*\circ \ell_{E^\prime}^{(r)}$ is a consequence
of Proposition \ref{prop:ul-ell}(4) and naturality of $exp$ with respect to an inclusion of elementary subalgebras of
$\fg$.  This implies that the colimit $\ell_\tau^{(r)} $ is well defined.  Since each $\ell_E^{(r)}$ is graded and
since the grading on $V_r(\bG)$ is functorial with respect to $\bG$, we conclude that $\ell_\tau^{(r)} $ is a map of 
graded $k$-algebras.  Since each $\ell_E^{(r)}$ factors through the projection $X_E \to Y_E$ by Proposition 
\ref{prop:ul-ell}, we conclude that $\ell_\tau^{(r)} $ factors through the projection $X_\tau \to Y_\tau$.

To show that $\ell_\tau^{(r)} $ is $\tau$-equivariant we verify the commutativity for each $x \in \tau$ of the
following diagram
 \begin{equation}
 \label{diag:tau-equi}
\begin{xy}*!C\xybox{%
\xymatrix{
X_E \ar[d]_{c_x} \ar[r]^{\ell_E^{(r)}} & V_r(\ul E) \ar[d]^{c_x} \ar[r]^{i_{E*}} & V_r(\bG) \ar[d]^{c_x}  \\
X_{E^x}  \ar[r]^{\ell_{E^x}^{(r)}} & V_r(\ul E^x)  \ar[r]^{i_{E^x*}} & V_r(\bG) 
 }
}\end{xy}
\end{equation}
where $c_x$ denotes conjugation by $x$.  The commutativity of the left square is seen by examining the
definition of $\ell_E^{(r)}$ in (\ref{eqn:ell-r}), whereas the commutativity of the right square follows by
naturality of $V_r(-)$ and the equality $c_x \circ i_E = i_{E^x} \circ c_x$.

Since $\ell_\tau^{(r)}$ is a map of graded $k$-algebras, it induces $\bP  \ell_\tau^{(r)}: \bP Y_\tau \ \to \ \Proj k[V_r(\bG)].$
Taking the colimit of the factorizations of each  $(\ell_E^{(r)})^*$ given in the proof of Proposition \ref{prop:ul-ell}
provides the asserted factorization of $\bP \ell_\tau^{(r)}$.
\end{proof}

 We now extend Proposition \ref{prop:E-compare} to $\bG$.
 The theorem below sharpens relationships between the restrictions to $\bG(\bF_p)$
 and to $\bG_{(1)}$ of a rational $\bG$-module established for example in \cite{LN}, \cite{CLN} and later 
investigations of the relationships between the restrictions of $M$ to $\bG(\bF_p)$ and $\bG_{(r)}$.
Thanks to Corollary \ref{cor:inj-restr} and the construction of $\ell_\tau^{(r)}$ in Proposition \ref{prop:cln-extend},
the proof follows ``immediately" from Proposition \ref{prop:E-compare}.
 
 \begin{thm}
\label{thm:Lie-compare}
Assume that $\bG$ satisfies the conditions of Hypothesis \ref{hypoth:G}.  Let $M$ be a finite
dimensional rational $\bG$-module whose restriction to $\bG_{(r)}$ has non-zero constant $j$-rank
and which has exponential degree $< p^r$.  Thus,  the restriction of $M$ to $\tau \equiv \bG(\bF_p)$ has constant $j$-rank 
equal to that of $M$ restricted to $\bG_{(r))}$.

The classes in $K_0(\bU_\tau)$  of the $(r-1)$-st Frobenius twists of the vector bundles 
$$ker\{ (\bP \Theta_{\tau,M})^j \}, \quad coker\{ (\bP \Theta_{\tau,M})^j \}, \quad im\{ (\bP \Theta_{\tau,M})^j\} $$
on $\bU_\tau$ constructed in Theorem \ref{thm:projbundle} are equal to the images under
$$(\bP \ell_\tau^{(r)})^*: K_0(\Proj k[V_r(\bG)]) \ \to \ K_0(\bU_\tau)$$ 
of the classes of the corresponding vector bundles 
$$ker\{ \bP \Theta_{\ul E_{(r)},M}^j  \}, \quad 
coker\{ \bP \Theta_{\ul E_{(r)},M}^j  \}, \quad im\{ \bP \Theta_{\ul E_{(r)},M}^j  \}$$
on $\Proj k[V_r(\bG)] $  constructed in \cite{FP4}. 
\end{thm}

\begin{proof}
Because $M$ has exponential degree $< p^r$, $M$ has constant $j$-rank for $\tau$ whenever $M$ has constant $j$-rank for $\bG_{(r)}$ by \cite[Cor 4.14]{F2}.

Corollary \ref{cor:inj-restr} tells us that the natural map $K_0(\bU_\tau) \to \bigoplus_{i=1}^s K_0(\bU_{E_i})$ is 
injective, where $\{ E_1,\ldots,E_s \}$ are the maximal elements of $\cE(\tau)$.   Naturality of our constructions
implies that it suffices to verification the asserted equalities for each $E_i$ in place of $\tau$.  This is the 
content of Proposition  \ref{prop:E-compare}.
\end{proof}

\begin{remark}
The correspondence in Proposition \ref{prop:elem} is not associated to maps of (finite) group schemes, rather uses an
identification of algebras of the non-isomorphic Hopf algebras $kE$ and $k\bG_a^{\times s}$.  The necessity of 
considering higher order Frobenius kernels as in Theorem \ref{thm:Lie-compare} is particularly evident
when considering a rational $\bG$-module $M$ which is the Frobenius twist $M= N^{(1)}$ of another rational $G$-module $N$.
In this case, the restriction of $M$ to $\bG_{(1)}$ is trivial.
\end{remark}

\vskip .2in


\section{Considerations for $G\rtimes \tau$}
\label{sec:Gtau}

Throughout this section, $G$ will denote an infinitesimal group scheme (i.e., a group scheme whose coordinate
algebra is a local, finite dimensional $k$-algebra whose maximal ideal is nilpotent)  of height $\leq r$ for some 
$r > 0$ and $\tau$ will denote a finite group.  If $k$ is algebraically closed, then any finite group 
scheme over $k$ is of the form $G \rtimes \tau$, where
$G$ is the connected component of the finite group scheme and $\tau$ is its group of connected components.
In this section, we associate vector bundles to modules of constant $j$-type for such 
semi-direct products $G \rtimes \tau$.

We first consider the case that the action of $\tau$ on $G$ is trivial; in other words, we consider $G\times \tau$.

We recall the map of $k$-algebras
\begin{equation}
\label{eqn:epsilon}
\epsilon: k[t]/t^p \ \to \ k\bG_{a(r)} = k[u_0,\ldots,u_{r-1}]/(u_i^p), \quad t \mapsto u_{r-1};
\end{equation}  
this is a Hopf algebra map if and only if $r=1$.
For any $K$-point $\psi$ of the affine scheme $V_r(G)$ corresponding to the 1-parameter subgroup 
$\psi: \bG_{a(r),K} \to KG$, we associate the following $\pi$-point of $G$
$$\alpha_\psi: K[t]/t^p \ = \  K\bG_{a(1)} \ \stackrel{\epsilon}{\to}
K\bG_{a(r)} \ \stackrel{\psi_\alpha}{\to} \ KG.$$
(Here, and in what follows, we denote by $kG$ the $k$-linear dual of the coordinate algebra $k[G]$
of $G$.)

For any $K$-point $(\psi,u)$ of $V_r(G) \times X_\tau$, we associate 
\begin{equation}
\label{eqn:alpha-psi-u}
\alpha_{\psi,u}: K[t]/t^p \ \to \ K(G\times \tau), \quad t \mapsto \alpha_\psi(t) \otimes 1 + 1 \otimes \alpha_u(t)  
\in KG \otimes_K KE \ \subset K(G\times \tau,)
\end{equation}
where $E \subset \tau$ is some elementary abelian $p$-subgroup such that 
$u \in J_E\otimes K$ is a $K$ point of $X_E = \bA_{J_E} \ \subset \ X_\tau$ and 
$\alpha_u: K[t]/t^p \to KE$ sends $t$ to $u$.  

\begin{defn}
\label{defn:Gtimestau}
We set
$$A_{G\times \tau} \ \equiv \  k[V_r(G)] \otimes A_\tau, \quad B_{G\times \tau} \ \equiv \  k[V_r(G)] \otimes B_\tau$$
so that 
$$X_{G\times \tau} \ \equiv \Spec A_{G\times \tau} \ =  \ V_r(G) \times X_\tau, \quad 
Y_{G\times \tau} \ \equiv \Spec B_{G\times \tau} \ =  \ V_r(G) \times Y_\tau.$$
Similarly, we define
$$A_{G\times \tau}^{(2)} \ \equiv \  k[V_r(G)] \otimes A_\tau^{(2)}, \quad 
X_{G\times \tau}^{(2)} \ \equiv \Spec A_{G\times \tau}^{(2)}.$$
We recall that $k[V_r(G)]$ admits a natural grading (see \cite{SFB1}) and thus so does $A_{G\times \tau}$.

We  define 
$$\bP X_{G\times \tau} \ \equiv \ \Proj(A_{G\times \tau} ), \quad \bU_{G\times \tau} \ \equiv \ 
\Proj A_{G\times \tau} \backslash \Proj A_{G\times \tau}^{(2)}.$$
Finally, we equip
$$X_{G\times \tau}, \quad X^{(2)}_{G\times \tau}, \quad Y_{G\times \tau}, 
\quad X_{G\times \tau} \backslash X^{(2)}_{G\times \tau}; \quad
\bP X_{G\times \tau}, \quad \bP X^{(2)}_{G\times \tau}, \quad \bP Y_{G\times \tau}, \quad \bU_{G\times \tau}$$
with the actions of $\tau$ determined by the trivial action on $V_r(G)$ and the action on $X_\tau$ determined by 
the action of $\tau$ on the partially ordered set $\cE(\tau)$ of elementary abelian $p$-subgroups of $\tau$.
\end{defn}

\vskip .1in

We recall the natural map of finitely generated $k$-algebras $\psi: H^\bu(G,k) \ \to \ k[V_r(G)]$ of \cite{SFB1} which is a
``$p$-isogeny" as proved in \cite[Thm 5.2]{SFB2}; in other words, every element of the kernel of $\psi$ 
has some $p$-th power 0 and every element is $k[V_r(G)]$ has some $p$-th power in the image of $\psi$.  This
implies that the morphism on prime spectral $\Spec k[V_r(G)] \to \Spec H^\bu(G,k)$ induced by $\psi$ is a homemorphism.
We also recall that Quillen's theorem \cite[Thm 7.1]{Q1} asserts that the restriction map $\rho: H^\bu(\tau,k) \ \to \ (\varprojlim_{E \in \cE(\tau)} H^*(E,k))^\tau = (A_\tau)^\tau$ is a $p$-isogeny.  

The tensor product of the  $p$-isogenies $\psi: H^\bu(G,k) \ \to \ k[V_r(G)]$ and 
$\rho: H^\bu(\tau,k) \ \to \ (\varprojlim_{E \in \cE(\tau)} H^*(E,k)^\tau)$ determines 
$p$-isogenies
$$(A_{G \times \tau})^\tau = k[V_r(G)] \otimes (A_\tau)^\tau \ \leftarrow \ H^\bu(G,k) \otimes H^\bu(\tau,k) 
\ \to \ H^\bu(G\times \tau,k),$$
which induce a natural homeomorphism $\Spec H^\bu(G\times \tau,k) \ \stackrel{\approx}{\to} \ (Y_\tau/\tau) \times V_r(G).$

These $p$-isogenies are maps of graded algebras, thereby determining morphisms
\begin{equation}
\label{eqn:proj-homeo}
\Pi(G\times \tau) \simeq \Proj H^\bu(G\times \tau,k) \ \to \ \Proj (H^\bu(G,k) \otimes H^\bu(\tau,k)) \ \leftarrow
(\bP Y_{G\times \tau})/\tau
\end{equation}
which are homoemorphisms.

\begin{defn}
\label{defn:Theta-Gtimestau}
Define
$$\Theta_{G\times \tau} \ = \ p_1^*(\Theta_{G}) + p_2^*(\Theta_\tau^{(r-1)}) \  \in \ 
(k[V_r(G)] \otimes kG) \otimes (A_\tau\otimes k\tau) \simeq A_{G\times \tau} \otimes k(G\times \tau),$$
where $\Theta_G \in k[V_r(G)] \otimes kG$ is the universal $p$-nilpotent operator constructed in \cite{FP4}, 
$\Theta_\tau \in A_\tau \otimes k\tau$ is given in Definition \ref{defn:Theta-tau}, and $\Theta_\tau^{(r-1)} = F^{r-1*}(\Theta_\tau)$
is the image of $\Theta_\tau$ under the map $F^{r-1} \times 1: A_\tau \otimes k\tau \to A_\tau \otimes k\tau$  (in
other words, the $(r-1)$-st Frobenius twist of $\Theta_\tau$).

For any $k(G\times \tau)$-module $M$, define the $A_{G\times \tau}$-linear endomorphism of $A_{G\times \tau} \otimes M$
$$\Theta_{G\times\tau,M}: A_{G\times \tau} \otimes M \to A_{G\times \tau} \otimes M, \quad 1\otimes m \mapsto 
\Theta_G(m) + \Theta_\tau^{(r-1)}(m)$$
and similarly define the $\cO_{\bP X_{G\times \tau}}$-linear map of sheaves on $\Proj(A_{G\times \tau})$ by 
$$\bP \Theta_{G\times \tau,M}: \cO_{\bP X_{G\times \tau}} \otimes M \ \to \ (\cO_{\bP X_{G\times \tau}} \otimes M)(p^{r-1}),
\quad m \to \bP \Theta_G(m) + \bP \Theta_\tau^{(r-1)} (m).$$
(See \cite[Prop 2.1]{FP4} for the graded degree of $\Theta_G$.)
\end{defn}
\vskip .1in 

We next proceed to extend Proposition \ref{prop:Theta-tauM} to finite group schemes of the form $G\times \tau$.

\begin{prop}
\label{prop:Theta-GtimestauM}
The specialization of the action of $\Theta_{G\times\tau,M}$ on $M \otimes A_{G\times \tau}$ along some geometric point \\
$\xi = (\psi,u): A_{G\times\tau} \to K$ is given by
$$\Theta_{G\times\tau,M,\xi}: K\otimes M \ \to  \ K \otimes M, \quad  m \mapsto (\alpha_{\psi,u}(t))\cdot m 
=  \alpha_\psi(t)\cdot m + \alpha_{F^{r-1}(u)}(t) \cdot m.$$

For any $k(G\times k\tau)$-module $M$, $\Theta_{G\times\tau,M}$ is $\tau$-equivariant and has $p$-th power equal to 0. 
\end{prop}

\begin{proof}
The identification of the specialization of the action of $\Theta_{G\times\tau,M}$ follows immediately from the facts that the
specialization of the action of $\Theta_{G,M}$ on $M\otimes k[V_r(G)]$ at $\psi$ is given by multiplication by 
$\alpha_\psi(t)$ and the specialization of the action of $\Theta_{\tau,M}$ on $M\otimes A_\tau$ at $u \in J_E \subset X_\tau$
is given by multiplication by $u$.

Since $\tau$ acts trivially on $G$ (and therefore on $KG$ and $k[V_r(G)]$) and since 
$\Theta_{\tau,M}$ is $\tau$-equivariant by Propositon \ref{prop:Theta-tauM}, 
we conclude that $\Theta_{G\times\tau,M}$ is also $\tau$-equivariant for this is a ``linear
extension" of the sum of $\Theta_{\tau,M}$ and $\Theta_{G\times\tau,M}$.

Since elements of $G$ commute with elements of $\tau$ in $G\times \tau$, we conclude that the (Zariski) localization
of $(\Theta_{G\times \tau,M})^p$ at any scheme point $\xi = (\psi,u): A_{G\times\tau} \to k(\psi,u)$ is given by the 
sum of multiplication by the Zariski localization of $(\Theta_{G,M})^p$  and of $(\Theta_{\tau,M}^{(r-1)})^p$, both of which are 0.
\end{proof}

Since $\Theta_{G \times \tau}$ is homogeneous of degree $p^{r-1}$ in $A_{G\times \tau} \otimes k(G\times \tau)$
viewed as a graded module over the (commutative) graded algebra $A_{G\times \tau}$ with $k(G\times \tau)$
placed in degree 0, $\Theta_{G \times \tau}$ determines 
$$\bP \Theta_{G\times \tau}: \cO_{\bP X_{G\times \tau}} \ \to \  \cO_{\bP X_{G\times \tau}}(p^{r-1})$$
and restricts to
$$\bP \Theta_{G\times \tau}: \cO_{\bU_{G\times \tau}} \ \to \  \cO_{\bU_{G\times \tau}}(p^{r-1})$$

The following theorem is the evident generalization of Theorem \ref{thm:projbundle}; the proof of the latter theorem
applies to prove Theorem \ref{thm:Gtimes-tau} with only minor 
 ``notational changes" replacing $\bU_\tau, \ \bP X_\tau, \ \Theta_\tau, \ A_\tau$ by their generalizations
$\bU_{G\times \tau}, \ \bP X_{G\times \tau}, \ \Theta_{G\times \tau}, \\ A_{G\times \tau}$, and 
adjusting references accordingly.

\begin{thm}
\label{thm:Gtimes-tau}
Let $\tau$ be a finite group and $G$ an infinitesimal group scheme $G$ of height $\leq r$.
If $M$ is a finite dimensional $G \times \tau$-module of constant $j$-rank, then the restrictions 
to $\bU_{G\times \tau}$ of the $\tau$-equivariant coherent sheaves on $\bP X_{G\times \tau}$
$$ker\{ (\bP \Theta_{G\times \tau,M})^j\}, \quad coker\{ (\bP \Theta_{G\times \tau,M})^j\}, 
\quad im\{ (\bP\Theta_{G\times \tau,M})^j\} $$
are $\tau$-equivariant vector bundles on $\bU_{G\times \tau}$.
\end{thm}

We now consider a general semi-direct product $G\rtimes \tau$ with $G$ infinitesimal and $\tau$ finite.
 We utilize the limits $\varinjlim_{(E^\prime, E) \subset \tau}$ and 
$\varprojlim_{(E^\prime, E) \subset \tau}$ indexed by the partially ordered subset of 
pairs $(E^\prime \subset E)$ of elementary abelian $p$-subgroups of $\tau$; a map of such
pairs $(F^\prime \subset F) \to (E^\prime \subset E)$ is a quadruple 
$F^\prime \subset E^\prime \subset E \subset  F$.  
Observe that $\tau$ acts on this indexing category, with $g \in \tau$
sending  $(E^\prime \subset E)$ to $((E^\prime )^g\subset E^g)$.  

We consider
the natural map from pairs $(E^\prime, E) \subset \tau$ to subgroup schemes of $G\rtimes \tau$:
$$(E^\prime, E) \subset \tau \ \ \mapsto \ \ G^E \times E^\prime,$$
where $G^E \subset G$ is the centralizer of $E$ in $G$.   Since $G^E \subset G$
is also infinitesimal, it is connected.  Observe that 
$$(F^\prime \subset F) \to (E^\prime \subset E) \quad \Rightarrow \quad G^F \times F^\prime \ 
\hookrightarrow \ G^E \times E^\prime.$$

Proposition 4.12 of \cite{FP2} (see also Corollary 5.4 of \cite{FP1}) asserts that the $\Pi$-point space
$\Pi(G\rtimes \tau)$ (isomorphic as schemes to $\Proj H^\bu(G\rtimes \tau,k)$) is given by the following:
$$(\varinjlim_ {(E^\prime, E) \subset \tau} \Pi(G^E \times E^\prime))/\tau \ \stackrel{\sim}{\to} \ \Pi(G\rtimes \tau).$$
This suggests extending Theorem \ref{thm:Gtimes-tau} by using Theorem \ref{thm:Gtimes-tau}
and applying $\varinjlim_{(E^\prime, E) \subset \tau}(-)$, which we now proceed to do.

\begin{prop}
\label{prop:FP2}
Consider $G\rtimes \tau$ with $G$ infinitesimal of height $\leq r$ for some $r > 0$ and $\tau$ finite.
The limits
$$A_{G\rtimes \tau} \ \equiv \varprojlim _{(E,E^\prime)\subset \tau} A_{G^E \times E^\prime}, \quad
B_{G\rtimes \tau} \ \equiv \varprojlim _{(E,E^\prime)\subset \tau} B_{G^E \times E^\prime}, \quad
A^{(2)}_{G\rtimes\tau} \ \equiv \varprojlim _{(E,E^\prime) \subset \tau} A^{(2)}_{G^E \times E^\prime}$$
are finitely generated commutative $k$-algebras equipped with actions of $\tau$ and natural gradings.  These algebras are
the coordinate algebras of the affine schemes representing the indicated colimits of schemes
$$X_{G \rtimes \tau} \simeq \varinjlim_{(E,E^\prime)\subset \tau} X_{G^E \times E^\prime} , \quad
Y_{G \rtimes\tau} \simeq \varinjlim_{(E,E^\prime)\subset \tau} Y_{G^E \times E^\prime} , \quad
X_{G\rtimes\tau}^{(2)} \simeq \varinjlim_{(E,E^\prime)\subset \tau} X^{(2)}_{G^E \times E^\prime}.$$
Applying $\Proj(-)$, we obtain the following schemes (the first two of which are projective)
$$\bP X_{G\rtimes \tau}, \quad \bP Y_{G\rtimes \tau}, \quad \bU_{G\rtimes \tau} \ \equiv \ \bP X_{G\rtimes \tau}\backslash 
\bP X_{G\rtimes \tau}^{(2)} \simeq
\varinjlim_{(E,E^\prime)\subset \tau} \bU_{G^E\times E^\prime}.$$

The natural maps 
$$p_{G\rtimes \tau}: X_{G\rtimes \tau} \ \to \ Y_{G\rtimes \tau}; 
\quad p_{G\rtimes \tau}: \bU_{G\rtimes \tau} \ \to \ \bP Y_{G\rtimes \tau}$$
are $\tau$-equivariant.

Furthermore, there is a natural $p$-isogeny $ \bP Y_{G\rtimes \tau}/\tau\ \stackrel{\simeq}{\to} \ \Pi(G\rtimes \tau).$
\end{prop}

\begin{proof}
We consider the colimit $\varinjlim_{(E,E^\prime)\subset \tau} V_r(G^E) \times A_{J_{E^\prime}}$, constructed inductively
by push-out squares as follows.   We let $\{ E_{i,1}, \ldots E_{i,n_i} \}$ denote the set of  the elementary abelian 
subgroups of $\tau$ of rank $i$.   We begin with $V_r(G)$ and take the pushout of 
$$V_r(G) \times \{ 0 \} \ \leftarrow \ \coprod_{j=1}^{n_1} V_r(G^{E_{1,j}}) \times \{ 0 \} \ \to \  
\coprod_{j=1}^{n_1} V_r(G^{E_{1,j}}) \times \bA_{J(E_{1,j})},$$
which we denote by $C_1$.  By attaching each
$V_r(G^{E_{1,j}}) \times \bA_{J(E_{1,j})}$ one at a time, we may express $C_1$ as a succession of pushouts with both
``left arrow" and ``right arrow" a closed immersion.
We proceed inductively with the respect to $i$ to construct $C_{i+1}$.  Namely, we view $C_{i+1}$ as the pushout
(for a fixed $i$) of 
$$C_i  \ \leftarrow \ \coprod_{E_{i-i,j^\prime} \subset E_{i,j}} V_r(G^{E_{i,j}}) \times \bA_{J(E_{i-1,j^\prime})}\ \to 
\ \coprod_{j=1}^{n_i} V_r(G^{E_{i,j}}) \times \bA_{J(E_{i,j})}.$$
We proceed by induction on $j, \ 1 \leq j \leq n_j$, viewing $C_{i+1}$ as a colimit of pushouts
$C_{i,j}$ defined as the pushouts (now for a fixed $i,j$)
$$C_{i,j-1} \ \leftarrow \ \coprod_{E_{i-i,j^\prime} \subset E_{i,j}} V_r(G^{E_{i,j}}) \times \bA_{J(E_{i-1,j^\prime})} \ \to 
\  V_r(G^{E_{i,j}}) \times \bA_{J(E_{i,j})}.$$
We identify this pushout as the pushout (once again for fixed $i,j$) of
\begin{equation}
\label{eqn:finalcolimit}C_{i,j-1} \ \leftarrow \ \varinjlim_{F \in \cE(E_{i,j})^\prime} V_r(G^{E_{i,j}}) \times \bA_{J(F)}  
\ \to \  V_r(G^{E_{i,j}}) \times \bA_{J(E_{i,j})},
\end{equation}
where $\cE(E_{i,j})^\prime$ is the partially ordered subset of proper subgroups of $E_{i,j}$.  The colimit in 
(\ref{eqn:finalcolimit}) is shown to be representable and the map to $V_r(G^{E_{i,j}}) \times \bA_{J(E_{i,j})}$ is
verified to be a closed embedding using a simple version of the argument Proposition \ref{prop:limits-colimits}
which in turn uses Theorem \ref{thm:ferrand}.  The left map of (\ref{eqn:finalcolimit}) is the natural map from the
colimit to $C_{i,j-1}$ which is recursively seen to be a closed immersion.

Consequently, proceeding step by step along this iterated induction context, we conclude that 
$$X_{G \rtimes \tau} \ \equiv \ \varinjlim_{(E,E^\prime)\subset \tau} X_{G^E \times E^\prime} \ = 
\ \varinjlim_{(E,E^\prime)\subset \tau} V_r(G^E) \times A_{J_{E^\prime}}$$ 
is representable, constructed as the 
spectrum of the inverse limit of corresponding coordinate algebras of $V_r(G^E) \times A_{J_{E^\prime}}$.

The representability of $\varinjlim_{(E,E^\prime)\subset \tau} Y_{G^E \times E^\prime}, \ 
\varinjlim_{(E,E^\prime)\subset \tau} X_{G^E \times E^\prime}^{(2)}$ is proved in a completely similar fashion. 
The construction of the map $p_{G\rtimes \tau}: X_{G\rtimes \tau} \to Y_{G\rtimes \tau}$ follows easily from 
the functorial property of colimits.  
The associated projective schemes $\bP X_{G\rtimes \tau}, \ \bP Y_{G\rtimes \tau}$ and
the map $\bU_{G\rtimes \tau} \to \bP Y_{G\rtimes \tau}$ are constructed using the gradings as in Section \ref{sec:tau-equiv}.
The asserted $\tau$-equivariance of $p_{G\rtimes \tau}$ follows from the $\tau$-equivariance of $p_\tau$.

Finally, the fact that the natural map $(\varinjlim_{(E,E^\prime)\subset \tau} \Pi(G^E \times E^\prime))/\tau \ \stackrel{\sim}{\to} \ \Pi(G\rtimes \tau) $ is a $p$-isogeny
is essentially given in \cite[5.4]{FP1} (see also \cite[4.12]{FP2}).   On the other hand, 
the $p$-isogeny $\Pi(G^E \times E^\prime)  \ \stackrel{\sim}{\to} \
\Proj(k[V_r(G)] \otimes S^\bu(J_{E^\prime}/J_{E^\prime}^2))$ follows from \cite[Thm 5.2]{SFB2}.  
We identify $( \varinjlim_{(E,E^\prime)\subset \tau}  \Proj(k[V_r(G)] \otimes S^\bu(J_{E^\prime}/J_{E^\prime}^2)))/\tau$
with $\bP Y_{G\rtimes \tau}/\tau$.  Thus, composing the quotient by $\tau$ of the colimit of the second $p$-isogeny
with the first $p$-isogeny determines
$\bP Y_{G\rtimes \tau}/\tau \ \stackrel{\simeq}{\to} \  \Pi(G\rtimes \tau).$
\end{proof}

\begin{defn}
\label{defn:Theta-Grtimestau}
Consider the semi-direct product $G\rtimes \tau$ where $G$ is an infinitesimal group scheme over $k$ and 
$\tau$ is a finite group.  We define 
$$\Theta_{G\rtimes \tau} \ \equiv \  \varprojlim_{(E,E^\prime)\subset \tau} \Theta_{G^E \times E^\prime} \ \in \
\varprojlim_{(E,E^\prime)\subset \tau}A_{G^E\times E^\prime} \otimes k(G^E \otimes kE^\prime).$$

For any finite dimensional $k(G\rtimes\tau)$-module $M$, we define
$$\Theta_{G\rtimes \tau,M} \ \equiv \  \varprojlim_{(E,E^\prime)\subset \tau} \Theta_{G^E \times E^\prime,M}:
M \otimes A_{G\rtimes \tau} \ \to \ M \otimes A_{G\rtimes \tau}$$
as the limit of the maps $\Theta_{G^E \times E^\prime,M}: M \otimes A_{G^E\times E^\prime} \to  M \otimes A_{G^E\times E^\prime}$.
Here, we have identified $M \otimes A_{G\rtimes \tau}$ with $\varprojlim_{(E,E^\prime)\subset \tau}M \otimes A_{G^E\times E^\prime}$.
\end{defn}

\vskip .1in

\begin{prop}
\label{prop:Theta-Gtimestau}
For any $k(G\rtimes \tau)$-module $M$, $\Theta_{G\rtimes \tau,M}$ determines an $\cO_{\bP X_{G \rtimes \tau}}$-linear map
$$\bP \Theta_{G\rtimes \tau,M}: \cO_{\bP X_{G\times \tau}} \otimes M \ \to \ (\cO_{\bP X_{G\rtimes \tau}} \otimes M)(1)$$ 
which is $\tau$-equivariant and 
whose $p$-th iterate is 0.
\end{prop}

\begin{proof}
We view $\Theta_{G\rtimes \tau,M}$ as a map of the graded $A_{G\rtimes\tau}$-module $M \otimes A_{G\rtimes\tau}$ to the 
graded module $M \otimes A_{G\rtimes\tau}[1]$, thereby determining $\bP \Theta_{G\rtimes \tau,M}$.  Since
$\Theta_{G\rtimes \tau,M}$ is the limit of graded maps of $A_{G^E\times E^\prime}$-modules, 
$$\bP \Theta_{G\rtimes \tau,M} \ = \ \varprojlim_{(E,E^\prime)\subset \tau} \bP \Theta_{G^E\times E^\prime,M} $$
Since each $\bP \Theta_{G^E\times E^\prime,M}$ has $p$-th power is 0, the $p$-th power of  $\bP \Theta_{G\rtimes \tau,M}$
is also 0.

To prove that $\Theta_{G\rtimes \tau,M}$ is $\tau$-equivariant, we adapt the proof of Proposition \ref{prop:Theta-tauM}
for $\Theta_{\tau,M}$.  For $m \in M$ we write
$$\Theta_{G\rtimes \tau,M}(m) \ = \ \varprojlim_{(E,E^\prime)\subset \tau} 
\sum \phi_{G^E}\otimes f_{E^\prime} \otimes (X_{G^E} \otimes u_{E^\prime})(m),$$
where for each $(E,E^\prime)$ the sum is indexed by finitely many 4-tuples
$$\phi_{G^E} \in k[V_r(G^E)], \ f_{E^\prime} \in S^\bu(J_{E^\prime}^*), \ X_{G^E} \in k(G^E)), \
u_{E^\prime} \in J_{E^\prime}.$$
Applying $g \in \tau$ to $\Theta_{G\rtimes \tau,M}(m)$ gives
$$\varprojlim_{(E,E^\prime)\subset \tau} \sum g(\phi_{G^E}) \otimes g(f_{E^\prime}) \otimes g\cdot 
(X_{G^E} \otimes u_{E^\prime})(m)$$ 
which equals $\Theta_{G\rtimes \tau,M}(g(m))$ (after reindexing).

The $\tau$-equivariance of  $\Theta_{G\rtimes \tau,M}$ implies the $\tau$-equivariance of $\bP \Theta_{G\rtimes \tau,M}$.
\end{proof}

We provide one last construction of vector bundles, this time associated to finite dimensional 
$G \rtimes \tau$-module of constant $j$-rank.

\begin{thm}
\label{thm:proj-Grtimestau}
Consider the finite group scheme $G \rtimes \tau$, with $G$ infinitesimal of height $\leq r$ and $\tau$ a finite
group.   If $M$ is a finite dimensional $G \rtimes \tau$-module of constant $j$-rank, then the restrictions 
to $\bU_{G\rtimes \tau}$ of the coherent sheaves on $\bP X_{G\rtimes \tau}$
$$ker\{ (\bP \Theta_{G\rtimes \tau,M})^j\}, \quad coker\{ (\bP \Theta_{G\rtimes \tau,M})^j) \}, \quad 
im\{ (\bP \Theta_{G\rtimes \tau,M})^j\} $$
$\tau$-equivariant are vector bundles on $\bU_{G\rtimes \tau}$. 
\end{thm}

\begin{proof}
Since $\bU_{G\rtimes \tau} \ = \ \varinjlim_{(E,E^\prime)\subset \tau} \bU_{G^E\times E^\prime}$, it suffices to observe that each 
of $ker\{ (\bP \Theta_{G\rtimes \tau,M})^j\}, \ coker\{ (\bP \Theta_{G\rtimes \tau,M})^j) \},\
im\{ (\bP \Theta_{G\rtimes \tau,M})^j\} $ restrict to vector bundles on each $\bU_{G^E\times E^\prime}$ by 
Theorem \ref{thm:Gtimes-tau} and apply Milnor's patching argument \cite[Thm 2.1]{Mil}.
\end{proof}

\end{document}